\documentclass{article}
\usepackage[affil-it]{authblk}
\usepackage{amsthm} 
\usepackage{varioref}
\usepackage{amssymb}
\usepackage{amsmath}
\usepackage{tikz}
\usepackage{hyperref}
\providecommand{\abs}[1]{\lvert#1\rvert}
\providecommand{\norm}[1]{\lVert#1\rVert}
\usepackage{color}
\usepackage[latin1]{inputenc}
\usepackage[T1]{fontenc}
\usepackage{pifont}

\usepackage{comment}

\usepackage{enumerate}
\usetikzlibrary{arrows}
\usetikzlibrary{calc}

\usepackage{abstract}
\usepackage{authblk}
\usepackage{blindtext}

\usepackage{lipsum}
\newcommand\blfootnote[1]{%
  \begingroup
  \renewcommand\thefootnote{}\footnote{#1}%
  \addtocounter{footnote}{-1}%
  \endgroup
}

\usepackage[titletoc,title]{appendix}
\usepackage{ledmac}

\def\XXint#1#2#3{{\setbox0=\hbox{$#1{#2#3}{\int}$ }
\vcenter{\hbox{$#2#3$ }}\kern-.6\wd0}}

\begin{document}

\numberwithin{equation}{section}
\newtheorem{cor}{Corollary}[section]
\newtheorem{df}[cor]{Definition}
\newtheorem{rem}[cor]{Remark}
\newtheorem{theo}[cor]{Theorem}
\newtheorem{pr}[cor]{Proposition}
\newtheorem{lem}[cor]{Lemma}
\newtheorem{conj}[cor]{Conjecture}

\renewcommand{\)}{\right)}
\newcommand{\mR}{{\mathbb R}}
\newcommand{\mE}{{\mathbb E}}
\newcommand{\mF}{{\mathbb F}}
\newcommand{\mC}{{\mathbb C}}

\renewcommand{\a}{{\alpha}}
\renewcommand{\o}{{\overline{\Omega}}}
\renewcommand{\l}{\left(}
\renewcommand{\r}{\right)}

\title{  Standard Planar Double Bubbles are Stable \\
under Surface Diffusion Flow}

\author{Helmut Abels, 
  Nasrin Arab\thanks{Corresponding author.},
 Harald Garcke}
 \affil{Fakultät für Mathematik, Universität Regensburg,
\\       93040 Regensburg, Germany}   
\date{\today}

\maketitle

\begin{abstract}
Although standard planar double bubbles are stable in the sense that the  second variation of the  perimeter functional is non-negative  for all area-preserving perturbations the question arises whether they are dynamically stable. By presenting connections between these two concepts of stability  for double bubbles, we prove that  standard planar double bubbles are stable under the surface diffusion flow via the generalized principle of linearized stability in parabolic Hölder spaces.   
\end{abstract}
 

\section{Introduction}
\blfootnote{E-mail addresses:
\sf{helmut.abels@mathematik.uni-regensburg.de }(H. Abels), \\
\sf{ } \sf{ nasrin.arab@mathematik.uni-regensburg.de}, \sf{nasrin.arab@gmail.com} (N. Arab),\\
\sf{harald.garcke@mathematik.uni-regensburg.de} (H. Garcke).}The standard double bubble is stable in the sense that the second variation of the area  functional is non-negative. This follows for example from the fact that it is a local minimum of the area functional under volume constraints. It is however an open problem whether
the standard double bubble is stable for volume conserving geometric flows such as the surface diffusion flow.

The  related  problem for one bubble has been studied by Escher, Mayer and Simonett, see  \cite{Escher-Mayer-Simonett,Escher-Simonett}, who showed  that spheres are stable under the surface diffusion flow and  the volume preserving mean curvature flow. 
In this paper we show that  the  standard double bubble in $\mR^2$ is stable under the surface diffusion flow. In case of equal areas the result is illustrated in Figure \ref{Fig: stability}. 

 Before moving on to   define the problem  more precisely, let us make one point clear: Consider a  (cost) functional having  local minimizers.  Even though minimizers exist it is not clear that an associated gradient flow will converge to these minimizers, see \cite{AbsilKurdyka} for ODE examples. In other words, if a stationary state of the associated gradient flow is a local minimum, this
 in general does not imply stability of  this equilibrium under the flow. 

As just mentioned, the surface diffusion flow is the volume preserving gradient flow of the area functional. Indeed, it is the fastest way to decrease area  while preserving the volume w.r.t. the $H^{-1}$-inner product; see e.g. \cite{Mayer-SD,Taylor-Cahn,HG}. Let us now  define the flow precisely. A surface is evolving in time under  \textit{the surface diffusion flow} if its normal velocity  is equal to the negative surface Laplacian of its mean curvature at each point, that is, if a surface  $\Gamma(t) $  satisfies    
\begin{equation} \label{Eq: Intro surface}
V(t) = -\Delta_{\Gamma(t)} H_{\Gamma(t)} \,.
\end{equation}
Here $V$ stands for the normal velocity, $H$ is the mean curvature, and $\Delta $ is the Laplace-Beltrami operator, of the surface $\Gamma(t)$. 
Surfaces with constant mean curvature are  stationary solutions of the flow \eqref{Eq: Intro surface}. 
This flow leads to a fourth order parabolic partial differential equation (PDE). Thereby one may try to use  PDE theories to answer the question on the stability of  stationary solutions. 
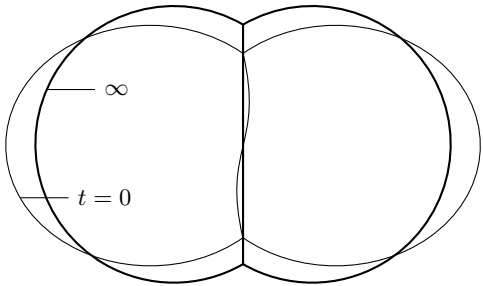
\begin{figure}[htbp]
        \centering
\begin{tikzpicture}[scale= 1.6,>=stealth]
\draw (0,0) arc (50:310:1.2cm and 1cm);

\draw (0,0) arc (50:310:-1.2cm and 1cm);

\draw (0,0) .. controls (0.18,-0.77) and (-0.18,-0.77) .. (0,-1.54);

\draw[ thick] (0 , 0.24 ) arc (60:300:1.15) ;
\draw[ thick] (0,0.24 ) arc (120:-120:1.15);
\draw[thick] (0,0.24)  -- (0, -1.76) ;
\draw (-1.85,-1.2) -- (-1.45,-1.2)node[ right ]{\small ${ t = 0}$};
\draw (-1.63,-0.3) -- (-1.23,-0.3)node[ right ]{\small ${ \infty}$};
\end{tikzpicture}
 \caption{An illustration  of  the stability of standard planar double bubbles, possibly up to isometries. (cf. the cover page to G. Prokert's PhD thesis  \cite{Prokert})} \label{Fig: stability}
\end{figure} 

Recently Prüss, Simonett and Zacher \cite{Pruss20093902,pruss2009612} introduced a    practical tool to show  stability for evolution equations in infinite dimensional Banach spaces in cases where the  linearization has a non-trivial kernel.  It is called \textit{the generalized principle of linearized stability}. This principle is extended in \cite{AbelsArabGarcke} to cover a more general setting. 
According to this principle, to prove stability, one   needs to verify  four assumptions known as the conditions of \textit{normal stability}:
\begin{enumerate}[\upshape (i)]
\item
the  set of stationary solutions  creates locally a  smooth manifold of  finite dimension, 
\item
the tangent space of the manifold of stationary solutions is given by the null space of the linearized operator,
\item
the eigenvalue $0$ of  the linearized operator  is  semi-simple,
\item
apart from zero, the spectrum of the linearized operator lies in $\mathbb{C}_+$.
\end{enumerate}     
We will see that the non-negativity of the second variation of the area functional  plays an important role in verifying most of these assumptions for the double bubble problem.

Let us note that the center manifold theory  is used in \cite{Escher-Mayer-Simonett,Escher-Simonett}  to prove the stability of spheres under the surface diffusion flow and the volume preserving mean curvature flow.
 We remark that sofar no center manifold theory exists in the case of non-homogenous boundary conditions. Due to the  triple junctions, we indeed get  nonlinear boundary conditions in the corresponding PDE.
 
\textit{Outline.}  In Section  \ref{Sec: the geometric setting} we precisely define     the problem  which we summarize here: Let $\Gamma^0$ be an initial planar double bubble. We suppose that    $\Gamma^0$ moves according to the surface diffusion flow including certain boundary conditions on the triple junctions. We continue then by observing that the set of stationary solutions  consists precisely  of all   standard planar double bubbles.  

Next we transfer, via suitable parameterization, this geometric problem to   a system of fully nonlinear and nonlocal partial differential equations with nonlinear boundary conditions defined on fixed domains. We then linearize this nonlinear system. This is   done in Section \ref{Sec: PDE and Linear}.    

In Section \ref{Sec: general setting} we rewrite this nonlinear system  as a perturbation of the linearized  problem. We  then see  how suitably the problem fits to the generalized principle of linearized stability setting  which is summarized in Section \ref{setting}. 

It then remains to check the conditions of normal stability. Let us note here that understanding the geometric interpretations of the problem was of great help.
 Lemma \ref{Lem: spectrum negative} proves  assertion (iv). The non-negativity of the second variation is the main ingredient  in the  proof.  Semi-simplicity is also proved by the non-negativity of the second variation in Section \ref{Sec: semi-simplicity2}. 
 We prove assertion (i) in Section \ref{Sec: D manifold}  and Corollary \ref{cor: null and manifold} proves assertion (ii). 

By applying the generalized principle of linearized stability we then complete the proof of the stability, as summarized in Section \ref{Sec: main theo}. 
We continue  then in Section \ref{Sec: general flow}  to discuss   general  area preserving geometric flows. We then  conjecture that the standard planar double bubbles are stable under sufficiently smooth  area preserving gradient flows, see Conjecture \ref{Conj}.     

In addition, Appendix \ref{Sec: more about bilinear}   shows that the second variation is  negative for two elements of the basis of the  null space which correspond to non-area preserving perturbations.

\textit{\textbf{Acknowledgments. }}Part of this work was carried out while Arab was visiting Freie Universität Berlin the geometric analysis group    during the summer and  the winter semester in  2014. Arab was supported by Bayerisches Programm zur Realisierung der Chancengleichheit für Frauen in Forschung und Lehre und nationaler MINT-Pakt as well as  DFG, GRK 1692 "Curvature, Cycles, and Cohomology". The support is gratefully acknowledged.

\section{The geometric setting} \label{Sec: the geometric setting}
 A planar double bubble $\Gamma \subset \mR^2$    consists of three   curves $\Gamma_1, \Gamma_2, \Gamma_3$ meeting  two common points $p_+, p_-$ (triple junctions) at their boundaries such that $\Gamma_1$ and $\Gamma_2$ (resp. $\Gamma_2$ and $\Gamma_3$) enclose the connected region $R_1$ (resp. $R_2$). Hence the curve $\Gamma_2$ is the curve separating $R_1$ and $R_2$, see Figure \ref{Fig: planar double bubble}.

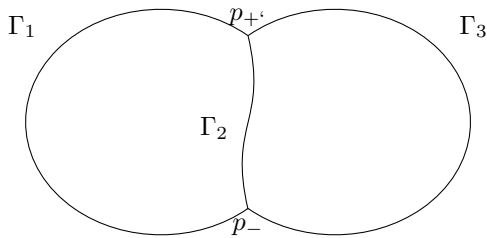
\begin{figure}[htbp]
        \centering
\begin{tikzpicture}[scale= 1.5,>=stealth]
\draw (0,0) arc (50:310:1.2cm and 1cm);

\draw (0,0) arc (50:310:-1.2cm and 1cm);

\draw (0,0) .. controls (0.18,-0.77) and (-0.18,-0.77) .. (0,-1.54);
\coordinate [label=above:{$p_{+`}$}] (A) at (0,0);
\coordinate [label=below:{$p_{-}$}] (B) at (0,-1.54);

\coordinate [label=above:{$\Gamma_1$}] (A) at (-2,-0.1);
\coordinate [label=above:{$\Gamma_3$}] (A) at (2,-0.1);
\coordinate [label=above:{$\Gamma_2$}] (A) at (-0.3,-1);
\end{tikzpicture}
 \caption{A good example of a planar double bubble $\Gamma = \{ \Gamma_1, \Gamma_2, \Gamma_3 \}$}
 \label{Fig: planar double bubble}
\end{figure}
We study the following problem introduced by Garcke and Novick-Cohen \cite{GarckeNovick}: Find evolving  planar double bubbles $\Gamma(t) = \{ \Gamma_1(t), \Gamma_2(t), \Gamma_3(t) \}$   with the following properties:
\begin{equation} \label{Eq: 2-double bubble}
\left.
   \begin{aligned}
       V_i                                            &= -\Delta_{\Gamma_i} \kappa_i \quad   & &\text{on } \Gamma_i (t) \,, \\
        \sphericalangle ( \Gamma_1 (t) , \Gamma_2 (t)) &=\sphericalangle ( \Gamma_2 (t) , \Gamma_3 (t)) = \sphericalangle ( \Gamma_3 (t) , \Gamma_1 (t)) = \tfrac{ 2 \pi}{3}        & &\text{on } \Sigma(t) \,,   \\
       \kappa_{1} + \kappa_2 + \kappa_3                              &= 0          & &\text{on } \Sigma(t) \,, \\
       \nabla_{\Gamma_1} \kappa_1 \cdot n_{\partial \Gamma_1} &=\nabla_{\Gamma_2}       \kappa_2 \cdot n_{\partial \Gamma_2}                                                     =\nabla_{\Gamma_3}       \kappa_3 \cdot n_{\partial \Gamma_3}  & &\text{on } \Sigma (t) \,,\\
      \Gamma_i(t)|_{t=0}                                &=\Gamma_i^0 \,,     \end{aligned}
\right \}   
\end{equation}
where $i = 1, 2, 3$, $\Gamma_i(t) \subset \mR^2 $, and
$$
\partial \Gamma_1 (t) = \partial \Gamma_2 (t) = \partial \Gamma_3 (t)\, \big( = \{p_+(t), p_-(t) \} =: \Sigma(t)\big) \,.
$$ 
Here $V_i$ is the normal velocity, $\kappa_i$  is the curvature, and $\Delta_{\Gamma_i}$ is the Laplace-Beltrami operator of the curve $\Gamma_i$ $(i=1,2,3)$. Also $\nabla_{\Gamma_i}$ denotes the surface gradient and $n_{\partial \Gamma_i}$ denotes the outer unit conormal of $\Gamma_i$  at $\partial \Gamma_i $ $( i = 1,2,3)$.

Moreover $\Gamma^      0 = \{ \Gamma^0_1, \Gamma^0_2, \Gamma^0_3 \}$ is a given initial planar double bubble, which fulfills the 
 angle $\eqref{Eq: 2-double bubble}_2$, the curvature  $\eqref{Eq: 2-double bubble}_3$ and the balance of flux condition $\eqref{Eq: 2-double bubble}_4$ as above and    satisfies the compatibility condition  
\begin{equation} \label{Con: sum laplace curvature}
\Delta_{\Gamma_1^0} \kappa_1^0 + \Delta_{\Gamma_2^0} \kappa_2^0 + \Delta_{\Gamma_3^0} \kappa_3^0 = 0 \qquad \text{ on } \Sigma(0) \,.
\end{equation}
Furthermore, the choice of unit normals $n_i(t)$ of $\Gamma_i(t)$ is illustrated in Figure \ref{Fig.normals}, which in particular determines the sign of curvatures $\kappa_1$, $\kappa_2$ and $\kappa_3$.
We say  that the curve has positive curvature if it is curved in the direction of the normal. 
\begin{figure}[htbp]
           \centering
        \begin{tikzpicture}[scale=0.4,>=stealth]
                \draw (2.828,2.828) arc (45:315:4) node[above=2.8cm,right=-2.5cm]{$ \Gamma_1$};
                \draw (2.828,2.828) arc (105:-105:2.928)node[above=2cm,right=1.3cm]{$ \Gamma_3$};
                \draw (2.828,2.828) arc (165:195:10.92) node[above=1cm,right=-0.7cm]{$ \Gamma_2$};
             
                 \draw [->] (2.828,2.828) -- (3.828,3.828) node[above]{$n_1$};
                
                 \draw [->] (2.828,2.828) --(3.25,1.4) node[right]{$n_3$};
          
                    \draw[->] (2.828,2.828) -- (1.4,3.2) node[left]{$n_2$};
                        \end{tikzpicture}
                        \caption{The choice of the  normals}
                        \label{Fig.normals}
\end{figure}
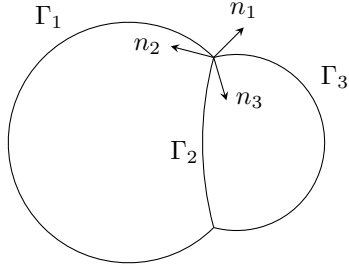

Let us   give a motivation for assuming the condition  \eqref{Con: sum laplace curvature}  on initial planar double bubble. 
\begin{lem}
For  a classical solution of the surface diffusion flow \eqref{Eq: 2-double bubble}  we have  
\begin{equation} \label{Eq: sum laplace curvature}
\sum_{i = 1}^3 \Delta_{\Gamma_i} \kappa_i = 0 \qquad \text{ on } \Sigma(t) \,.
\end{equation}
\end{lem}
\begin{proof}
At the triple junctions $p_\pm(t)$ we can write  for the normal velocities 
$$
V_i = \big\langle\frac{d}{d \tau} p_\pm ( \tau) \Big|_{\tau = t} , n_i(t) \big \rangle \,. 
$$
  Now the angle condition  implies 
$$
\sum_{i = 1}^3 V_i = \sum_{i = 1}^3 \big\langle\frac{d}{d \tau} p_\pm ( \tau) \Big|_{\tau = t} , n_i(t) \big \rangle = \big\langle\frac{d}{d \tau} p_\pm ( \tau) \Big|_{\tau = t} ,  \sum_{i = 1}^3 n_i(t) \big \rangle = 0 \,.
$$
As  $V_i = \Delta_{\Gamma_i} \kappa_i$, we obtain  \eqref{Eq: sum laplace curvature}.
\end{proof}
Therefore if one seeks for a classical solution which  is continuous  up to the time $t = 0, $ one should impose  \eqref{Eq: sum laplace curvature} on the initial data. 

After introducing the problem,   let us see    its interesting      geometric\ properties:
\begin{lem}
A classical solution to the surface diffusion flow \eqref{Eq: 2-double bubble} decreases the total length and preserves the enclosed areas.
\end{lem}
\begin{proof}
Assume $\Gamma(t) $ is a solution to the flow \eqref{Eq: 2-double bubble} and let 
$$  
l(t) = \sum_{i=1}^3 \int_{\Gamma_i(t)}1 \,\mathrm{d} s 
$$
denote the total length.
A transport theorem (see e.g. \cite[Theorem 2.44]{Depner}) gives: \vspace{-0.2cm}
\begin{align} \label{Formula: length}
\frac{d}{dt} l(t) &= -\sum_{i = 1}^3 \ \int_{\Gamma_i(t)} V_i \, \kappa_i \,\, \mathrm{d}s +  \int_{\Sigma(t)} \overbrace{\sum_{i=1}^3 \nu_{\partial \Gamma_i}}^{=0} \nonumber = \sum_{i =1}^3 \int_{ \Gamma_i(t)} (\Delta_{\Gamma_i(t)} \kappa_i ) \kappa_i \, \mathrm{d} s \nonumber \\
&=- \sum_{i =1}^3 \int_{\Gamma_i(t)} |\nabla_{\Gamma_i(t)} \kappa_i |^2  \mathrm{d} s + \int_{\Sigma(t)} \sum_{i =1}^3 (\nabla_{\Gamma_i(t)} \kappa_i \cdot n_{\partial_{\Gamma_i}}) \kappa_i \nonumber \\
&=  - \sum_{i =1}^3 \int_{\Gamma_i(t)} |\nabla_{\Gamma_i(t)} \kappa_i |^2  \mathrm{d} s +  \int_{\Sigma(t)}  (\nabla_{\Gamma_1(t)} \kappa_1 \cdot n_{\partial_{\Gamma_1}}) \underbrace{ \sum_{i =1}^3 \kappa_i }_{=0} \nonumber\\[-0.52cm]
&=- \sum_{i =1}^3 \int_{\Gamma_i(t)} |\nabla_{\Gamma_i(t)} \kappa_i |^2  \mathrm{d} s \leq 0 \,, 
\end{align}
where we used all the boundary conditions. Note that the  sum of the normal boundary velocities  $\nu_{\partial \Gamma_i}$ vanishes due to the angle condition, more precisely,
$$
\sum_{i=1}^3 \nu_{\partial \Gamma_i}(t, p_{\pm}(t)) =\Big( \frac{d}{d \tau} p_{\pm}(\tau) \Big|_{\tau = t} \Big)\sum_{i =1}^3 n_{\partial \Gamma_i}(t, p_{\pm}(t)) = 0 \,. 
$$
Moreover, the integral over $\Sigma(t) = \{p_+(t), p_-(t) \}$ should be understood as a sum over its elements. 

Next, let us prove that the enclosed areas are preserved:
It is a standard fact that (see e.g.  \cite[equation (3.1)]{HutchingsMorganRitorAntonio})\begin{align*}
\frac{d}{d t} \int_{R_1(t)} 1\, \mathrm{d}x &=  \int_{\Gamma_1 (t)} V_1 \, \mathrm{d}s - \int_{\Gamma_2(t)} V_2 \,\mathrm{d}s\\
&=- \int_{\Gamma_1(t)} \Delta_{\Gamma_1(t)} \kappa_1  \, \mathrm{d} s + \int_{\Gamma_2(t)} \Delta_{\Gamma_2(t)} \kappa_2 \, \mathrm{d} s               \\
&=- \int_{\Sigma (t)} \nabla_{\Gamma_1(t)} \kappa_1 \cdot n_{\partial_{\Gamma_1(t)}}
+ \int_{\Sigma (t)} \nabla_{\Gamma_2(t)} \kappa_2 \cdot n_{\partial_{\Gamma_2(t)}} =0\,.
\end{align*}
 Similarly, we  get  $\frac{d}{d t} \int_{R_2(t)} 1\, \mathrm{d}x = 0$, which completes the proof.
\end{proof}
Let us mention that, via formally matched  asymptotic expansions, the flow \eqref{Eq: 2-double bubble} is derived as an singular limit of   a system of degenerate   Cahn-Hilliard equations in \cite{GarckeNovick}, where in particular the boundary conditions at each triple junction are derived. 
\subsection{ Equilibria} \label{Sec: Equilibria}
 Let a planar double bubble $\Gamma = \{ \Gamma_1, \Gamma_2, \Gamma_3\}$ be a stationary solution of the flow  \eqref{Eq: 2-double bubble}, i.e., $\Gamma$ satisfies  \eqref{Eq: 2-double bubble} with $V_i = 0$ for $i = 1,2,3$. As a consequence 
$$
\Delta_{\Gamma_i} \kappa_i = 0 \quad (i=1,2,3)\,.
$$   
By the same arguments used in \eqref{Formula: length} we get
$$
0 = \sum_{i =1}^3 \int_{ \Gamma_i}(\Delta_{\Gamma_i} \kappa_i )\kappa_i \, \mathrm{d} s =- \sum_{i =1}^3 \int_{\Gamma_i} |\nabla_{\Gamma_i} \kappa_i |^2  \mathrm{d} s \,.  
$$ 
Thus $ \nabla_{\Gamma_i} \kappa_i = 0$ on $ \Gamma_i $. Therefore $\kappa_1, \kappa_2, \kappa_3$ are constant.
Summing up,  a planar double bubble $\Gamma$ is a stationary solution  of the  flow  \eqref{Eq: 2-double bubble}   if and only if   
\begin{description}
\item[(i)] 
the  curvatures $\kappa_i$ are constant, with $\kappa_1 + \kappa_2 + \kappa_3 = 0$, and 
\item[(ii)]
$ \sphericalangle ( \Gamma_i  , \Gamma_j ) = \tfrac{ 2 \pi}{3}$ on $\Sigma$ or equivalently  $\sum_{i=1}^3 n_{\partial \Gamma_i} = 0 $ on $\Sigma$. 
\end{description}
It will turn out that  the set of stationary solutions  consists precisely  of all   standard planar double bubbles:
\begin{df}
A standard planar double bubble   consists of three circular arcs meeting at their  boundaries at  $120$ degree angles. (Here, we interpret  a line segment as a circular arc too.)
\end{df}
We refer to Figure \ref{Fig: SPDB} for an example.
\begin{figure}[htbp]
           \centering
        \begin{tikzpicture}[scale=0.35,>=stealth]
                \draw (2.828,2.828) arc (45:315:4);
                \draw (2.828,2.828) arc (105:-105:2.928);
                \draw (2.828,2.828) arc (165:195:10.92) ;
     \end{tikzpicture}
     \caption{The standard planar double bubble}
     \label{Fig: SPDB}
\end{figure}
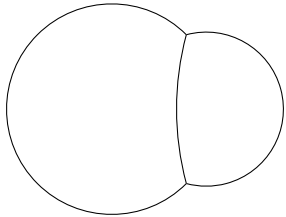
Indeed, as  circular arcs and line segments are the only curves with constant curvature, it  just remains to verify the condition on curvatures. This is done in the following proposition given in \cite[Proposition 2.1]{HutchingsMorganRitorAntonio}:\begin{pr}\label{Prop: existence SDB}
There is a unique standard planar double bubble (up to rigid motions, i.e., translations and rotations) for given areas in $\mR^{2 }$. The curvatures satisfy $\kappa_1 + \kappa_2 + \kappa_3 = 0$. 
\end{pr} 
\begin{rem}
As the choice of  the normals in \cite{HutchingsMorganRitorAntonio} differs from ours,  some sign differences  particularly for the curvature quantities  can occur.
\end{rem} 
Therefore the set of all standard planar double bubbles $DB_{r, \gamma , \theta}(a_1,a_{2})$ forms a $5$-parameter family
(see Figure \ref{Fig: D general PDB}), where
\begin{enumerate}[\upshape (i)]
\item 
$r > 0$ is the radius of  $\Gamma_1$, corresponding to scaling, 
\item
$(a_{1}, a_{2})$ is the center of $\Gamma_1$, corresponding to  translation,
\item
 the angle  $\theta$ corresponds to counterclockwise rotation around the center of $\Gamma_1$,
\item
the angle $0 < \gamma < \tfrac{2 \pi}{3}$ corresponds to the curvature ratio.
\end{enumerate} 
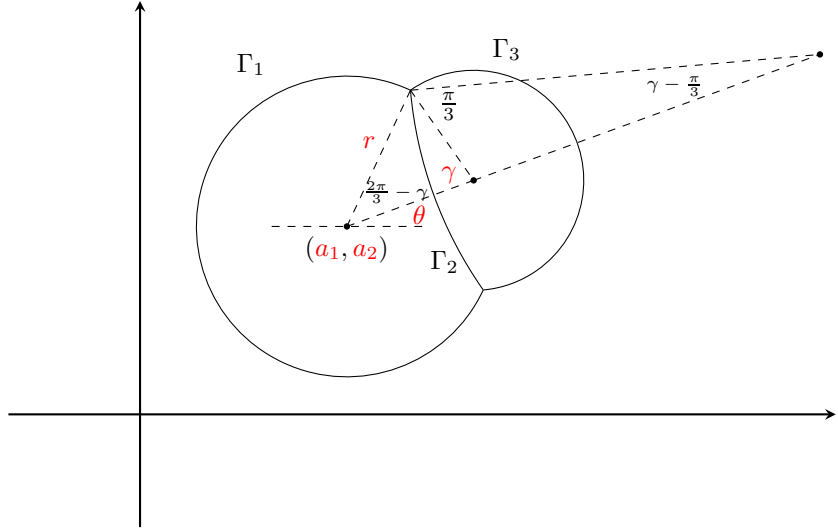
\begin{figure}[htbp]
     \centering
     \begin{tikzpicture}[scale=0.5,>=stealth]
                \draw (1.690 , 3.625 ) arc (65:335:4) node[above=3cm,right=-3.4cm]{$ \Gamma_1$};
                \draw (1.690 , 3.625 ) arc (125:-85:2.928) node[above=3.2cm,right= 0]{$ \Gamma_3$};
                \draw (1.690 , 3.625 ) arc (185:215:10.92) node[above= 0.37cm , left = 0.2cm]{$ \Gamma_2$};

                \filldraw [black] (0,0) circle (2pt) node [below]{$(\textcolor{red}{a_1},\textcolor{red}{a_2})$} node [above = 1.6 cm , right = 1.1cm ]{$\tfrac{ \pi} { 3 }$} node[above=13, right= 0]{
                                \begin{scriptsize}
                                        $\tfrac{2 \pi}{3}-\gamma  $ 
                                \end{scriptsize}
                           } ;
                \filldraw [black] (3.37,1.226) circle (2pt) node[above = 2, right =-0.55cm ]{$\textcolor{red}{\gamma}$};
                \filldraw [black] (12.577,4.577) circle (2pt) node [below = 12, left = 38]{
                        \begin{scriptsize}
                                $  \gamma - \tfrac{ \pi} { 3 } $
                        \end{scriptsize}
                };
                \draw[dashed] (0,0) -- (3.37,1.226);              
                \draw[dashed] (3.37,1.226) -- (12.577,4.577) ;   
                \draw[dashed]  (1.690 , 3.625 ) --(12.577,4.577);  
                \draw[dashed] (1.690 , 3.625 ) -- (0,0) node [above=32, right = 2.5]{$\textcolor{red}{r}$};           
                \draw[dashed] (1.690 , 3.625 ) -- (3.37,1.226) ;
                \draw [dashed](-2,0) -- (2,0) node[above=4.5, left= -5 ]{$\textcolor{red}{\theta} $};
                \draw [->,thick](-9,-5)--(13,-5);
                \draw [->,thick](-5.5,-8)--(-5.5,6);
    \end{tikzpicture}
                 \caption{The standard planar double bubble $\Gamma = DB_{r, \gamma , \theta}(a_1,a_{2})$}
                \label{Fig: D general PDB}
\end{figure}
Indeed, by the law of sines we have for $\gamma \neq \frac{ \pi}{3}$
\begin{equation} \label{Eq: D useful one}
\frac{\kappa_1}{\sin (\gamma + \frac{\pi}{3})} = \frac{\kappa_2}{\sin (\gamma - \frac{\pi}{3})} = \frac{\kappa_3}{\sin (\gamma - \pi )}
\end{equation}
and in case $\gamma = \tfrac{ \pi} { 3 }$ we observe
$ \kappa_2 = 0$ and $\kappa_1 = -\kappa_3$.
Note that due to  our choice of normals we always have $\kappa_1 < 0$ and $\kappa_3 > 0$.
Moreover,
$$ 
\begin{cases}
\kappa_2 > 0  &\text{ for } \gamma < \frac{\pi}{3}\,,\\
\kappa_2 < 0  &\text{ for } \gamma > \frac{\pi}{3}\,.
\end{cases} 
$$
For later use we define the constants $q_i$ as follows:
$$
q_i := -\frac{1}{\sqrt{3}} (\kappa_j - \kappa_k)
$$
for $(i,j,k) = (1,2,3), (2,3,1)$ and $(3,1,2)$.
Then the following result is true.
\begin{lem} 
We have$$
q_1 = \cot (\gamma + \tfrac{ \pi } {3})\kappa_1, \quad 
q_2 =
\begin{cases}
\cot (\gamma - \tfrac{ \pi } {3})\kappa_2 & \gamma \neq \frac{\pi}{3}\,,\\[10pt]
\dfrac{\kappa_1}{\sin(\frac{\pi}{3})}     & \gamma = \frac{\pi}{3} \,,
\end{cases}  \quad
q_3 = \cot (\gamma - \pi )\kappa_3 \,.
$$
\end{lem}
\begin{proof}
We calculate 
\begin{align*}
q_2 &= 
-\frac{1}{\sqrt{3}} (\kappa_3 - \kappa_1) =
-\frac{1}{\sqrt{3}} \Big( \frac{-\sin (\gamma ) -\sin (\gamma + \tfrac{ \pi}{3}) }{\sin(\gamma - \tfrac{ \pi }{3}) } \Big) \kappa_2  \\
&=\frac{2}{\sqrt{3}} \Big(  \frac{\sin ( \gamma +  \tfrac{ \pi} { 6 })  \cos (\tfrac{\pi}{6}) }{\sin(\gamma - \tfrac{ \pi }{3}) } \Big) \kappa_2= \cot( \gamma - \tfrac{ \pi}{3}) \kappa_2 \quad \text{for } \gamma \neq  \tfrac{\pi}{3}\,, \end{align*}
and obviously $q_2 = \frac{2}{\sqrt 3} \kappa_1 = \frac{\kappa_1}{\sin(\frac{\pi}{3})}  $ for $\gamma = \frac{\pi}{3}$. The continuity follows from   formula \eqref{Eq: D useful one}. The proof for $q_1$ and $q_3$ is  similar.
\end{proof}
Moreover, using the sum-to-product trigonometric  identity, we get  
\begin{equation} \label{cos and sin}
\left \{
\begin{aligned}
\sin ( \gamma + \frac{\pi}{3}) + 
\sin ( \gamma - \frac{\pi}{3}) +
\sin ( \gamma - \pi) &= 0 \,, \\
\cos ( \gamma + \frac{\pi}{3}) + 
\cos ( \gamma - \frac{\pi}{3}) +
\cos ( \gamma - \pi) &= 0 \,.
\end{aligned}
\right.
\end{equation}
One  strategy
to deal with geometric flows  on hypersurfaces is to parameterize the evolving hypersurfaces with
respect to a fixed reference hypersurface. This  eventually leads  to a PDE   on a fixed domain allowing us to employ  PDE theories. 

\section{ PDE formulation and linearization} \label{Sec: PDE and Linear}
In this section we introduce the proper setting  to reformulate the geometric flow \eqref{Eq: 2-double bubble} as a system of partial differential equations for  unknown  functions defined on  fixed domains. 
For this, we employ a parameterization with two parameters. The parameters correspond to a movement in  normal and tangential directions.  This parameterization  is adapted for two triple junctions from   Depner and Garcke \cite{DepnerGarckelinearized}, see also \cite{DepnerGarckeKohsaka}.
\subsection{Parameterization of planar double bubbles}

Let us describe $\Gamma_i(t)$ as  a graph over some fixed stationary solution  $\Gamma_i^* $  using  functions 
$$
\rho_i: \Gamma_i^*  \times [0,T)\rightarrow \mathbb R \qquad (i = 1,2,3)\,. 
$$
The precise way how $\rho_i$ defines $\Gamma_i(t)$ will be derived in what follows.

 Fix any stationary solution  
$$
\Gamma^* = DB_{r^*, \gamma^*, \theta^*}(a_1^*, a_2^*) \quad  \big(r^*>0,\, (a_1^*, a_2^*) \in \mR^2,\,  0 < \gamma^* < \tfrac{2 \pi}{3},\, 0 \leq \theta^* < 2 \pi \big) \,.
$$
Then  we observe  
\begin{align}
l_1^* &= -\tfrac{1}{\kappa_1^*} (  \gamma^* +\tfrac{\pi}{3} ),\nonumber \\
l_2^* &=
\left \{
\begin{aligned}
&-\tfrac{1}{  \kappa_2^* }(\gamma^* - \tfrac{\pi}{3}) = -\tfrac{1}{\kappa_1^*}\tfrac{( \gamma^{*} - \tfrac{\pi}{3})}{\sin (\gamma^{*} - \frac{ \pi}{3})}\sin (\gamma^* + \tfrac{\pi}{3})   & &\text{ if }\gamma^* \neq \tfrac{\pi}{3} \,,\\
&-\tfrac{1}{\kappa_1^*} \sin( \tfrac{\pi}{3}) & &\text{ if } \gamma^* = \tfrac{\pi}{3}\,,
\end{aligned}
\right.
  \nonumber\\
   l_3^* &=- \tfrac{1}{   \kappa_3^* }(\gamma^* - \pi) =  -\tfrac{1}{\kappa_1^*} \tfrac{(\gamma^* - \pi )}{\sin (\gamma ^*- \pi )} \sin (\gamma^* + \tfrac{\pi}{3}) \,,  \nonumber  
\end{align}
where $2 l^*_i$ is the length of $\Gamma_i^*$ $(i=1,2,3)$ and of course $ \kappa_1^*=- \frac{1}{r^*}$.

Let $\Phi_i^*:[-l^*_i, l_i^*] \to \mR^2$ be an arc-length parameterization of $\Gamma_i^*$. Hence 
$$
\Gamma^*_i=\left\{\Phi^*_i(x): x \in [-l^*_i, l^*_i]\right\}.
$$ 
Furthermore, set $({\Phi^*_i})^{-1}(\sigma)=x(\sigma)\in \mathbb R $,  for $\sigma\in \Gamma_i^*$. To simplify the presentation,   we hereafter  set 
\begin{equation} \label{DB:x and sigma}
\partial _\sigma w ( \sigma ):=\partial_x(w\circ \Phi_i^*) (x), \quad \sigma = \Phi_i^* (x),
\end{equation}
 that is,   we do not state the parameterization explicitly. Also we slightly  abuse    notation and write 
\begin{equation} \label{D: abuse of notation}
w(\sigma)=w(x)\quad (\sigma \in \Gamma_i^*)\,.
\end{equation}
To parameterize a curve nearby $\Gamma^*_i$, define 
\begin{align}
\Psi_i:\Gamma_i^*\times (-\epsilon,\epsilon)\times (-\delta,\delta)&\longrightarrow \mathbb R^2\,,\\
(\sigma,w,r)&\mapsto \Psi_i(\sigma,w,r):=\sigma+w\, n_i^*(\sigma )+r \, \tau_i^*(\sigma)\,.\nonumber
\end{align}
Here $\tau_i^*$ denotes a tangential vector field on $\Gamma_i^*$ having  support in a neighborhood of $\partial \Gamma_i^*$, which is equal to the outer unit  conormal $n_{\partial \Gamma_i^*}$ at $\partial \Gamma_i^*$.

Define then $\Phi_i = (\Phi_i)_{\rho_i,\mu_i}$ (we often omit for shortness the subscript $(\rho_{i},\mu_i)$) by 
\begin{equation} \label{par2}
\Phi_i:\Gamma_i^* \times [0,T) \rightarrow \mathbb R^2 \,, \quad \Phi_{i}(\sigma,t):=
\Psi_i(\sigma, \rho_i(\sigma,t),\mu_i(\mathrm{pr}_i(\sigma),t))\,,
\end{equation}
for the functions 
\begin{equation} \label{par1}
\rho_i:\Gamma_i^* \times [0,T) \rightarrow (-\epsilon, \epsilon)\,, \quad \mu_i: \Sigma^* \times [0,T) \rightarrow (-\delta,\delta)\,,
\end{equation}
where, similarly as before, $\Sigma^* = \partial \Gamma_i^*= \{p_+^*,p_-^*\}$.

The projection $\mathrm{pr}_i:\Gamma_i^* \rightarrow \Sigma^* $ is defined by imposing the following condition: The point $\mathrm{pr}_i( \sigma ) \in \partial \Gamma_i^*$ has the shortest distance on $\Gamma_i^*$ to $\sigma $. Clearly, in a small neighborhood of $\partial \Gamma_i^*$, the projection $\mathrm{pr}_i$ is well-defined and this is sufficient for us  since this projection is just used in the product $\mu_i(\mathrm{pr}_i(\sigma ),t)\tau_i^*(\sigma )$, where the second term vanishes outside a (small) neighborhood of $\partial \Gamma_i^*$. 

Now let us set, for small $\epsilon, \delta>0$ and fixed $t$, 
$$
(\Phi_i)_t: \Gamma_i^* \rightarrow \mathbb R^2, \quad (\Phi_i)_t(\sigma ):=\Phi_i(\sigma ,t) \quad \forall \sigma \in \Gamma_i^* 
$$
to finally define a new curve 
\begin{equation}\label{def: gamma rho}
\Gamma_{\rho_i,\mu_i}(t):=\mathrm{image}((\Phi_i)_t)\,.
\end{equation}
Observe that for $\rho_i \equiv 0 $ and $\mu_i \equiv 0$, the curve $\Gamma_{\rho_i,\mu_i}(t)$  coincides with $\Gamma_i^*$ for all $t$.

At each triple junction,  we have prepared for a movement in normal and tangential direction, allowing for an evolution of the triple junctions. Therefore,  we can now     formulate the condition, that the curves $\Gamma_i(t)$ meet at the triple junctions at their boundary by  
\begin{equation} \label{Con: triple}
 \Phi_1(\sigma,t) = \Phi_2( \sigma ,t ) = \Phi_3( \sigma ,t) \quad\text{ for } \sigma \in \Sigma^*, \, t\ge 0\,.
\end{equation}
 Next we prove that this condition leads to a linear dependency 
at the boundary points. As a result, nonlocal terms will eventually enter 
into  PDE-formulations of the geometric evolution problem.
\begin{lem} \label{Con:contact}
Equivalent to the equations \eqref{Con: triple} are the following conditions
\begin{equation}\label{Con: equi triple}
\left \{
\begin{aligned}
\mathrm{(i)}\quad \; \: 0&=   \rho_1 + \rho_2 + \rho_3  && \quad\text{ on } \Sigma^*  , \, \\
\mathrm{(ii)} \quad \mu_i &= -\tfrac{1}{\sqrt{3}}(\rho_j - \rho_k) && \quad\text{ on } \Sigma^*  ,    
  \end{aligned}
\right.
\end{equation}
for $(i,j,k) = (1,2,3), (2,3,1)$ and $(3,1,2)$. 
\end{lem}
Here  the  linear dependency (ii)   can be recast as the  matrix equation 
\begin{equation} \label{Eq: matrix rho}
\mu = \mathcal J \rho \quad \text{ on } \Sigma^* ,
\end{equation}
with the notations $\mu = ( \mu_1, \mu_2, \mu_3 )$, $\rho = ( \rho_1, \rho_2, \rho_3 )$ and the matrix
$$ 
\mathcal J = -\frac{1}{\sqrt {3}}
\begin{pmatrix}
 0  & 1 & -1 \\
 -1 & 0 & 1 \\
 1& -1 & 0
  \end{pmatrix}\,.
$$
\begin{proof}
First we prove that   \eqref{Con: triple} implies \eqref{Con: equi triple}.
 Using  the definition of $\Phi_i$, \eqref{Con: triple} can be rewritten as 
\begin{equation} \label{Eq: rho, mu}
\rho_i\, n_i^*+\mu_i \,n_{\partial \Gamma_i^*} = \rho_j\, n_j^*+\mu_j \,n_{\partial \Gamma_j^* } \quad \text{ on } \Sigma^* 
\end{equation}
for $ (i,j)= (1,2), (2,3) $.  By setting 
$$
q := \rho_1\, n_1^*+\mu_1 \,n_{\partial \Gamma_1^*} = \rho_2\, n_2^*+\mu_2 \,n_{\partial \Gamma_2^*} = \rho_3 \, n_3^* + \mu_3 \, n_{\partial \Gamma_3^*} \quad \text{ on } \Sigma^*
$$
we obtain $\rho_i =  \langle q, n_i^* \rangle $ for $i = 1,2,3$. Thus  the angle condition  for  $\Gamma^*$ gives
$$
\sum_{i=1}^3 \rho_i = \sum_{i = 1}^3 \langle q, n_i^* \rangle =\langle q,  \sum_{i = 1}^3 n_i^* \rangle = 0 \,. 
$$  
This proves $\mathrm{(i)}$. 
As a  result of  \eqref{Eq: rho, mu} we see further 
$$
\rho_ i\langle n_i^*, n_j^* \rangle + \mu_i \langle n_{\partial \Gamma_i^*}, n_j^*  \rangle= \rho_{j} \quad \text{ on } \Sigma^* \,.
$$
On the other hand the angle condition  implies  
$$
\langle n_{i}^*, n_j^* \rangle = \cos(\tfrac{2 \pi}{3})\,,  \quad \langle n_{\partial \Gamma_i^*},  n_j^*  \rangle = \cos (2 \pi - (\tfrac{2 \pi}{3}+\tfrac{ \pi}{2})) = -\sin ( \tfrac{2 \pi}{3}) \quad \text{ on } \Sigma^*
$$
for $(i, j) = (1,2), (2,3), (3,1)$.
Therefore using (i)
we conclude $$
\mu_i =- \tfrac{1}{s}(\rho _j - c \rho _i) = - \tfrac{1}{s}((1 + c)\rho_j + c \rho_k) = \tfrac{c}{s}( \rho_j - \rho_k)\,,
$$
where $s:= \sin( \tfrac{2 \pi}{3})$ and $c := \cos( \tfrac{2 \pi}{3})= -\tfrac{1}{2}$ and this yields  assertion (ii).
The proof of the converse statement is explicitly given in \cite[Lemma 2.3]{DepnerGarckelinearized}.
\end{proof}
Note that we  followed \cite{GarckeItoKohsakatriplejunction} in proving   statement (i), while  an easier proof is given here for  assertion (ii). Notice further that \eqref{Con: equi triple} easily implies \begin{equation} \label{Con: mu}
\mu_1 + \mu_2 + \mu_3 = 0   \quad \text{ on } \Sigma^* \,.
\end{equation}
\begin{rem}
Let us now note that it is within this  set, i.e., the set of   all planar double bubbles which can be described as the graph  over $\Gamma^*$, that we  will seek  a solution to the problem \eqref{Eq: 2-double bubble}. 
\end{rem}
Naturally, we  assume  also that the initial double bubble $\Gamma^0$ from \eqref{Eq: 2-double bubble} is  given as  a 
 graph over $\Gamma^*$, i.e.,
$$
\Gamma_i^0 = \{ \Psi_i \big( \sigma, \rho_i^0(\sigma), \mu_i^0(\mathrm{pr}(\sigma)) \big): \sigma \in \Gamma^*_i \}  \quad (i = 1, 2, 3)
$$
for some  function $\rho^0$.
Here $\mu^0 = \mathcal J  \rho^0$ on  $\Sigma^*$ as  $\Gamma^0$ is assumed to be a double bubble, i.e., the curves $\Gamma_i^0$ meet  two triple junctions at their boundaries.

\subsection{ Nonlocal, nonlinear parabolic boundary-value PDE} \label{Sec: DB nonlocal/linear}
The idea is to first derive  evolution equations for $\rho_i$ and $\mu_i$ which have to hold if the $\Gamma_i \,(i = 1,2,3)$ in \eqref{def: gamma rho} satisfy the condition \eqref{Con: triple} and solve  the surface diffusion 
flow \eqref{Eq: 2-double bubble} and then to make use of the linear dependency \eqref{Eq: matrix rho} in  deriving  evolution equations solely for the functions $\rho_i$.

As you may have  noticed,  nonlocal terms will appear in the formulations  since this linear dependency   \eqref{Eq: matrix rho} just holds at the boundary points.

Appendix \ref{App: nonlocal} provides for the   reader's convenience the derivation in detail. Indeed  a similar derivation is done  in \cite{AbelsArabGarcke}, which is originally given in \cite{GarckeItoKohsakatriplejunction},  \cite{DepnerGarckeKohsaka}. 
Therefore, let us present the final   system of fourth-order nonlinear, nonlocal PDEs for $t > 0$, $i = 1, 2, 3 $ and $j = 1, 2, \dots, 6$:
\begin{equation} \label{Eq: D nonlinear semifinal}
\left \{
\begin{aligned}
\partial_t \rho_i &= \mathfrak{F}_i( \rho_i ,  \rho|_{\Sigma^*}) \\
& \quad {}+ \mathfrak{B}_i ( \rho_i ,   \rho|_{\Sigma^*})\big(  \big \{ \mathcal J  \big( I - \mathfrak{B} ( \rho ,\rho|_{\Sigma^*}) \mathcal J \big)^{-1} \mathfrak{F}( \rho ,  \rho|_{\Sigma^*})  \big \} \circ \mathrm{pr}_i \big)_i  & &\text{ on } \Gamma^*_i, \\
0 &= \mathfrak G_j ( \rho )   & &\text{ on } \Sigma^*, 
\end{aligned}
\right.
\end{equation} 
with the initial conditions 
$$
\rho_i (\cdot, 0 ) = \rho_i^0 \text{ on } \Gamma_i^*\,,
$$
where in particular  $\mathfrak{F}_i( \rho_i ,  \rho|_{\Sigma^*})$ is a fourth-order nonlinear equation in $\rho_i$.    
\begin{rem}
Note that the price to pay for obtaining  equations solely for functions $\rho_i$ is the appearance of nonlocal terms, in particular  the nonlocal terms of highest-order (fourth-order) $\mathfrak{F}( \rho ,  \rho|_{\Sigma^*}) \circ \mathrm{pr_i}$ into the formulation. 
\end{rem}  As demonstrated at the beginning of Appendix \ref{App: nonlocal}, the functions ${\mathfrak F}_i, {\mathfrak B}_i,  {\mathfrak G}_j$  are rational functions  in the $\rho$-dependent variables,  with  nonzero denominators  in some neighborhood of $\rho\equiv 0$ in $C^1(\Gamma^*)$ (can   be inside of square roots  equalling to $1$ in some neighborhood of $\rho\equiv 0$ in $C^1(\Gamma^*)$, see the term $ \frac{1}{J_i}$).
\subsection{Linearization around the stationary solution} \label{Sec: D linearization}

The linearization of the surface diffusion equations and the angle conditions around the stationary solution $\rho \equiv 0$ are done in \cite[Lemma 3.2]{DepnerGarckelinearized} and \cite[Lemma 3.4]{DepnerGarckelinearized} respectively.
\begin{rem}
Note that the situation in \cite{DepnerGarckelinearized} is slightly different from ours, but nevertheless the results obtained there are applicable to our problem.
More precisely, the authors in \cite{DepnerGarckelinearized} consider the situation where, apart from the appearance of a triple junction, one has to deal with a fixed boundary. However, as they assume that the triple junction will not touch the outer fixed boundary, they  can use an explicit parameterization, exactly as ours, around a triple junction  and another  parameterization near the fixed boundary and finally  they compose them with the help of a cut-off function. Thus we can  use  their result for each triple junction.
\end{rem}

Therefore, taking into account  the linear dependency (ii) from Lemma \ref{Con:contact},  we get for the linearization of the nonlinear  problem \eqref{Eq: D nonlinear semifinal} around $\rho\equiv0 $ (that is, around the stationary solution $\Gamma^*$) the following linear  system for $i = 1, 2, 3$
\begin{equation} \label{Eq: D linear 1}
\partial _t \rho_i+   \Delta_{\Gamma_i^*} \big(  \Delta_{\Gamma_i^*} \rho_i + (\kappa ^*_i)^2 \rho_i \big) =0 \quad \text{ in } \Gamma_i^* \,, 
\end{equation}
with the boundary conditions on $\Sigma^*$
\begin{equation} \label{Eq: D linear BCs 1}
\left\{
 \begin{aligned}
 \rho_1+ \rho_2  + \rho_3 &=0 \,,  \\
   q^*_i\rho_i + \partial_{n_ {\partial\Gamma_i^*}}\rho_i &= q^*_j \rho_j + \partial_{n_ {\partial\Gamma_j^*}}\rho_j & &(i,j) = (1,2), (2,3), \\
\textstyle{\sum}_{i=1}^3\Delta_{\Gamma_i^*} \rho_i + (\kappa ^*_i)^2 \rho_i &= 0, \\ 
   \partial_{n_ {\partial\Gamma_i^*}} \big( \Delta_{\Gamma_i^*} \rho_1 + (\kappa_i^{*})^2 \rho_i  \big)                     &=\partial_{n_ {\partial\Gamma_j^*}} \big( \Delta_{\Gamma_j^*} \rho_j + (\kappa_j^{*})^2 \rho_j  \big)  \quad & &(i,j) = (1,2), (2,3),                      
 \end{aligned}
\right.
\end{equation}
where $$
q^*_i = -\frac{1}{\sqrt{3}} (\kappa^*_j - \kappa^*_k)
$$
for $(i,j,k) = (1,2,3), (2,3,1)$ and $(3,1,2)$.

 Let us  recall the parameterization (remember our abuse of  notation  \eqref{D: abuse of notation}) and  employ the following facts \begin{equation*}
\begin{aligned}
   \Delta_{\Gamma_i^*}\rho_i                 &= \partial_x^2 \rho_i    & & \text{for }x \in [-l^{*}_i,l_i^*] \,, \\
    \partial_{n_{\partial \Gamma_i^*}}\rho_{i} &=\nabla_{\Gamma_i^*}\rho_{i} \cdot n_{\partial \Gamma_i^*} \\
&= \partial_x \rho_i \,(T^*_i \cdot n_{\partial_{\Gamma_i^*}})=\pm         \partial_x \rho_i                                         & &\text{         at } x=\pm l_i^{*}\,, \\
    \kappa_{n_{\partial \Gamma_i^*}}            &=\kappa_i^*                  & & \text{         at }         x = \pm l_i^* \,, \\ 
\end{aligned}
\end{equation*}
where  $x$ is the arc length parameter of $\Gamma_i^*$ and denote by  $T^*_i$  the tangential vector of $\Gamma_i^*$. We can then  rewrite the linearized problem in terms of functions $ \rho_i:[-l^*_i, l^*_i] \times [0,T) \rightarrow \mR $ as 
$$
\partial _t \rho_i+  \partial_x^2 \big( \partial_x^2 + (\kappa ^*_i)^2 \big)\rho_i =0 \quad \text{ for } x\in [-l_i^{*} , l^{*}_i]
$$
with the boundary conditions 
\begin{equation}
\left\{
 \begin{aligned} \label{Eq: D linear BCs}
 \rho_1+ \rho_2  + \rho_3 &=0 \,,  \\
   q^*_1\rho_1 \pm \partial_x \rho_1 &=q_2^* \rho_2 \pm \partial_x \rho_2 = q^*_3 \rho_3 \pm \partial_x \rho_3 \,,    \\
 {\textstyle\sum_{i =1 }^3} \big( \partial_x^2 \rho_i + (\kappa ^*_i)^2 \rho_i \big) &=0 \,,  \\
   \partial_x \big( \partial_x^2 + (\kappa_1^{*})^2  \big)\rho_1                     &=\partial_x \big( \partial_x^2 + (\kappa_2^{*})^2 \big)\rho_{2} = \partial_x \big(\partial_x^2 + (\kappa_3^{*})^2 \big)\rho_{3} \,.                     
 \end{aligned}
\right.
\end{equation}
In the boundary conditions \eqref{Eq: D linear BCs}  we have omitted  the terms $\pm l^*_i$ in the functions $\rho_i$. That is, for instance the boundary condition $\rho_1 + \rho_2 + \rho_3 = 0 $ should be read as
$$
\rho_1 ( \pm l^*_1) + \rho_2 (\pm l^*_2) + \rho_3 ( \pm l^*_3) = 0 \,. 
$$ 
Furthermore, notice that the linearized problem  is completely local as, in particular, we linearized around a stationary solution.

Now we are in a position to look for a suitable PDE theory in order to answer the question of stability.
The \textit{generalized principle of linearized stability in parabolic Hölder spaces,} proved in \cite{AbelsArabGarcke}, see also \cite{Pruss20093902,pruss2009612},  provides the tool.
\section{The generalized principle of linearized stability in parabolic Hölder spaces}\label{setting}
  
In this section we present the practical tool, proved in \cite{AbelsArabGarcke}, for  proving the stability of equilibria of fully nonlinear parabolic systems with  nonlinear boundary conditions in situations
where the set of stationary solutions creates a $C^2$-manifold of finite dimension
which is normally stable.
The parabolic Hölder spaces are used as function spaces.

 Let $\Omega\subset\mR^n$ be a bounded domain of class $C^{2m + \a}$ for $m \in \mathbb N$, $0<\a<1$ with the  boundary $\partial \Omega$ .  Consider the  nonlinear boundary value problem 
\begin{equation}
\left\{
 \begin{aligned} \label{eq1}
   \partial _t u(t, x) + A (u(t,\cdot))(x) &= F(u(t,.)) (x),  &  &x\in \overline{\Omega}\,,         & &t>0\,, \\[0.1cm] 
   B_j (u(t,\cdot))(x)                     &= G_j(u(t,.))(x), &  &x\in\partial\Omega\,,
        & &j= {1,\dots, mN}\,,\\[0.1cm]
   u(0, x)                                 &= u_{0}(x),       &  &x\in \overline{\Omega}\,,         \\[0.1cm]
 \end{aligned}
\right.
\end{equation}
where $u : \o \times [ 0, \infty ) \rightarrow \mathbb{R}^N$. 
Here $A$ denotes a linear $2m$th-order partial differential operator having the form

\begin{equation*}
(Au)(x)=\sum_{|\gamma|\leq 2m}a_\gamma(x)\nabla^{\gamma}u (x)\,, \quad x\in\overline{ \Omega}\,,
\end{equation*}
and  $B_j$ denote   linear partial differential operators of order  $m_j$, i.e., 
\begin{equation*}
 (B_ju)(x)=\sum_{|\beta|\leq m_{j}}b^{j}_\beta(x)\nabla^{\beta}u (x)\,,\quad x\in \partial \Omega \,,\quad j= 1,\dots,mN\,, 
\end{equation*}
with $
0 \leq m_1 \leq m_2 \leq \cdots \leq m_{mN} \leq 2m-1 $. 

The coefficients $a_\gamma(x)\in \mathbb{R}^{N\times N}$, $b^{_{j}}_\beta(x)\in\mathbb{R}^N$. We assume that  
\begin{list}{}{\leftmargin=1.0cm\topsep=0.1cm\itemsep=0.0cm\labelwidth=1cm}
\item[(H2)\,]
the elements of the matrix $a_\gamma (x)$ belong to $C^\a(\o) $ and
\item[] 
the elements of the matrix $b^{j}_\beta(x)$ belong to $C^{2m+\a-m_j}(\partial \Omega).$ 
\end{list}
Concerning the fully
nonlinear terms $F$ and $G_j$, let us suppose 
\begin{list}{}{\leftmargin=1.0cm\topsep=0.2cm\itemsep=0.1cm\labelwidth=1cm}
\item[(H1)]
$F:B(0,R)\subset C^{2m}(\overline{\Omega})\rightarrow C(\overline{\Omega}) $ is $C^1$ with Lipschitz continuous derivative, $F(0)=0, F'(0)=0,$ and the restriction of $F$ to $B(0,R)\subset C^{2m+\a}(\overline\Omega)$ has values in $C^\a(\overline\Omega) $ and is continuously differentiable,
\item[]
$G_j:B(0,R)\subset C^{m_{j}}(\overline{\Omega})\rightarrow C(\partial\Omega) $ is $C^2$ with Lipschitz continuous  second-order derivative,
$G_j(0)=0, G_{j}'(0)=0,$ and the restriction of $G_{j}$ to $B(0,R)\subset C^{2m+\a}(\overline\Omega)$ has values in $C^{2m+\a-m_j}(\partial\Omega) $ and is continuously differentiable.
\end{list}
We set $B=(B_1,\dots, B_{mN})$ and $G=(G_1,\dots ,G_{mN})$.

We denote by $\mathcal{E}\subset B_{X_1}(0,R)$ the set of stationary solutions  of \eqref{eq1}, i.e.,
\begin{equation}\label{set of equilibria E}
u\in \mathcal{E}\iff  u\in B_{X_1}(0,R)\,, \: Au=F(u)\quad \mbox{in } \Omega , \quad Bu=G(u)\quad \mbox{on }\partial \Omega \,,
\end{equation}
where $X_1 = C^{2m + \a }(\o)$.
The   assumption (H1) in particular implies that 
$$
u ^*\equiv 0 \text{ belongs to } \mathcal{E}\,.
$$
Now the key assumption is  that near  $u^* \equiv 0$ the set of equilibria $\mathcal E$ creates a   finite dimensional $C^2$-manifold. In other words we assume: 
There is a neighborhood $U\subset \mR^k$ of $0\in U$, and a $C^2$-function $ 
\Psi:U\rightarrow X_1
$, such that 
\begin{alignat}{1}
   \bullet \quad &\Psi (U)\subset \mathcal{E} \text{ and } \Psi(0)=u^*\equiv         0 , \nonumber\\
   \bullet \quad & \text{the rank of $\Psi'(0)$ equals $k$.} \nonumber
\end{alignat}
Moreover, we at last require that there are no other stationary solutions near $u^* \equiv 0$ in $X_1$ than those given by $\Psi(U)$. That is we assume for some $r_1>0$,
$$\mathcal{E}\cap B_{X_1}(u^*,r_1)=\Psi(U) \, .
$$

The linearization of \eqref{eq1} at $u^* \equiv 0$ is given by the operator $A_0$ which is the realization of $A$ with homogeneous boundary conditions in $X=C(\o)$, i.e., the operator with domain
\begin{eqnarray}\label{linear operator}
\begin{array}{l}
D(A_0)=\Big\{ u\in C(\o)\cap\bigcap\limits_{1< p<+\infty }W^{2m,p}(\Omega):\,\ Au\in  X, \quad Bu=0 \text{ on } \partial \Omega\Big\},\\[20pt]
\quad A_0u= A u,\quad u\in D(A_0)\,.
\end{array}
\end{eqnarray}

Let  $\nu (x)$ denote the  outer normal of $\partial \Omega$ at $x \in \partial \Omega $.
We assume further the \textit{normality condition}: 
\begin{eqnarray}\label{normality condition}
\left\{
\begin{array}{l}
\text{for each } x\in \partial \Omega, \text{ the matrix } 
  \begin{pmatrix}
  \sum_{| \beta | = k}b^{j_1}_\beta(x)(\nu(x))^\beta\\ \vdots \\ \sum_{| \beta | = k}b^{j_{n_{k}}}_\beta(x)(\nu(x))^\beta
  \end{pmatrix}
  \text{is surjective},\quad\\ [0.9cm]
  \text{where } \{ \, j_i : i=1, \dots, n_k \, \}=\{ \, j : m_j = k  \, \} \,, 
\end{array}
\right.
\end{eqnarray}

Suppose at last the following first-order compatibility conditions  holds:
For $j$ such that $m_j=0$ and $ x\in \partial \Omega$
\begin{equation} \label{compatibility}
\left\{
 \begin{aligned}
    B u^0                  &= G( u^0 )\,, \\
    B_j( A u^0- F( u^0 )) &= G_j' (u^0)(A u^0-F(u^0)) \,.\\
 \end{aligned}
\right.
\end{equation}
\begin{theo}\label{theo01} 
\emph{(\cite[Theorem 3.1]{AbelsArabGarcke})}
Let $ u^* \equiv 0$ be a stationary solution  of \eqref{eq1}.  Assume  that the regularity conditions (H1), (H2), the Lopatinskii-Shapiro
condition, the strong parabolicity  and finally the normality condition  \eqref{normality condition} hold.    Moreover assume that   $u^*$ is normally stable, i.e., suppose that
\begin{enumerate}[\upshape (i)]
\item
near $u^*$ the set of equilibria $ \mathcal{E}$ is a $C^{2}$-manifold  in $X_1$ of dimension $k\in \mathbb{N}$,
\item
the tangent space of $ \mathcal{E}$ at $u^*$ is given by $N\l A_0\r$,
\item
the eigenvalue $0$ of $A_0$ is  semi-simple, i.e., $R\l A_0\r\oplus N \l A_0 \r= \MakeUppercase{X,}$
\item
$\sigma \l A_0 \r \setminus  \{ 0 \}\subset \mathbb{C}_+=\{z\in \mathbb{C}: \mbox{Re }z>0\}$.
\end{enumerate}
Then the stationary solution $u^*$ is stable in $ X_1$. Moreover, if $u^0$ is sufficiently close to $u^*$ in $X_1$ and  satisfies the compatibility conditions \eqref{compatibility}, then the problem \eqref{eq1} has a unique solution in the parabolic Hölder spaces, i.e., 
$$
 u \in C^{1+\frac{\alpha}{2m}, 2m+\alpha} \big( [0, \infty)\times\o \big)
 $$and approaches some  $u^\infty\in \mathcal{E}$ exponentially fast in $X_1$ as $t \to \infty$.
\end{theo}
\begin{rem}
We refer to   \cite[Section 2]{AbelsArabGarcke}   for the  definitions of Lopatinskii-Shapiro condition and the strong parabolicity as well as for a complete treatment.\end{rem}
In order to apply this theorem to prove stability, we must first show   that our nonlinear, nonlocal problem \eqref{Eq: D nonlinear semifinal} has the form \eqref{eq1}. We then  devote the rest of the paper  to show  that the  problem \eqref{Eq: D nonlinear semifinal} verifies  all  hypothesis of   Theorem \ref{theo01}. 
\section{Verifying the hypotheses of Theorem \ref{theo01}}
\subsection{General setting} \label{Sec: general setting}
If we  change the variables   by setting 
for each $i =  2, 3$
$$
x =  \frac{ \tilde x + l_1^*}{2 l_1^*}l_i^*+\frac{ \tilde x - l_1^*}{2l_1^*}l_i^* \qquad  \tilde x \in [-l_1^* , l_1^*] \,, 
$$
then we  easily can restate the nonlinear, nonlocal system \eqref{Eq: D nonlinear semifinal} as a perturbation of a linearized problem,  that is of the form \eqref{eq1}, with $\Omega = [-l_1^*, l_1^*]$,
$$
A \rho =   
\begin{bmatrix}
(\mathbf {l}_1^{})^4 &0 & 0\\[0.2cm]
0 & \: \: (\mathbf {l}_2^{})^4 & 0 \\[0.2cm]
0 &0 & (\mathbf {l}_3^{})^4 
\end{bmatrix}
 \partial_x^4\rho +
\begin{bmatrix}
  (\mathbf {l}_1^{}\kappa_1^*)^2 & 0 & 0\\[0.2cm ]
0 & ( \mathbf {l}_2^{}\kappa_2^*)^2 & 0 \\[0.2cm]
0 & 0 & (\mathbf {l}_3^{} \kappa_3^*)^2
\end{bmatrix}
\partial_x^2 \rho \,, 
$$ 
and 
\begin{align*}
B_1 \rho &= 
\begin{bmatrix}
  1 & 1 & 1\\
\end{bmatrix}
\rho \,, \\[5pt]
B_2 \rho &= 
 \pm
\begin{bmatrix}
  \mathbf {l}_1 & -\mathbf {l}_2 & \: \: \: 0 \:\: 
\end{bmatrix}
\partial_x\rho + 
\begin{bmatrix}
  q_1^* & -q_2^* & \: \: \:0  \: \:\\
\end{bmatrix}
\rho \,,\\ 
B_3 \rho &= 
\pm 
\begin{bmatrix}
 \, 0 & \quad\mathbf {l}_2 & -\mathbf {l}_3\\
\end{bmatrix}
\partial_x\rho +
\begin{bmatrix}
  \,0 & \quad q_2^* & - q_3^*\\
\end{bmatrix}
\rho   \,, \\[5pt]
B_4 \rho &=  
\begin{bmatrix}
  (\mathbf {l}_1^{})^2 & (\mathbf {l}_2^{})^2  & (\mathbf {l}_3^{})^2 \\
\end{bmatrix}
\partial_x^2\rho +
\begin{bmatrix}
  (\kappa_1^*)^2 & (\kappa_2^*)^2 & (\kappa_3^*)^2\\
\end{bmatrix}
\rho \,,\\[5pt]
B_5 \rho &=  
\begin{bmatrix}
  (\mathbf {l}_1)^3 & -(\mathbf {l}_2)^3 & \quad \: \;0 \quad\\
\end{bmatrix}
\partial_x^3\rho + 
\begin{bmatrix}
  \mathbf {l}_1^{}(\kappa_1^*)^2 & -\mathbf {l}_2^{}(\kappa_2^*)^2 \hspace{9 pt} & \: \: \: \; 0\! \quad \quad\\
\end{bmatrix}
\partial_x^{}\rho \,,\\ 
B_6 \rho &= 
\begin{bmatrix}
  \: \, \; 0 & \quad \; \; \;(\mathbf {l}_2)^3 & -(\mathbf {l}_3)^3\\
\end{bmatrix}
\partial_x^3\rho + 
\begin{bmatrix}
  \hspace{15pt}0& \hspace{24pt} \mathbf {l}_2^{}(\kappa_2^*)^2 & -\mathbf {l}_3^{}(\kappa_3^*)^2\\
\end{bmatrix}
\partial_x^{}\rho \,. 
\end{align*}
To simplify the presentation, we have dropped the tilde. Here 
$$
\rho_{}: [-l_1^*, l_1^*] \times [0, \infty) \to \mR^3,  \quad\rho_{} = ( \rho_1, \rho_2, \rho_3)^T 
$$
and the constants are given as $\mathbf{l}_i^{} := \dfrac{l^*_1}{l_i^*}$ $(i=1,2,3)$.
 
When  we write \eqref{Eq: D nonlinear semifinal} in the form of \eqref{eq1},  the corresponding $F$ is a smooth function  defined in some neighborhood of $0$ in $C^4( \o)$ having values in $C( \o )$. The reason is that, $F$ is   Fréchet-differentiable of arbitrary order in some neighborhood of $0$ (using the differentiability of composition operators, see e.g. Theorem 1 and 2 of \cite[Section 5.5.3]{RunstSickel}). The  same argument works for the corresponding functions $G_i$. We have obtained that   assumption (H1) is  satisfied.

Obviously, the operators $A$ and $B_j$ satisfy the smoothness assumption (H2) and  the operator $A$ is strongly parabolic.
Now Let us  check  that the Lopatinskii-Shapiro condition (LS) holds. To verify this,   for  $\lambda \in \overline{\mC_+}, \lambda \neq 0$, we consider the  following ODE 
\begin{equation} \label{Eq: D LS}
\left \{
\begin{aligned}
\lambda v_i ( y) + (\mathbf {l}_i)^4 \partial_y^4 v_i ( y ) &= 0, \quad (y > 0) \,,\\
v_1(0) + v_2(0) + v_3(0) &= 0 \,, \\
\mathbf {l}_1 \partial_y v_1(0) &= \mathbf {l}_2 \partial_y v_2(0) = \mathbf {l}_3 \partial_y v_3(0)\,,\\
\textstyle{\sum}_{i=1}^3 (\mathbf {l}_i)^2 \partial_y^2 v_i(0) &= 0 \,,\\
(\mathbf {l}_1)^3 \partial_y^3 v_1 (0)  &= (\mathbf {l}_2)^3 \partial_y^3 v_2(0) = (\mathbf {l}_3)^3 \partial_y^3 v_3 (0)
\end{aligned}
\right.
\end{equation}
and we  show that  $ v \equiv 0$ is the only classical solution that vanishes at infinity.
The energy methods provide a simple proof: We test the first line of the  equation \eqref{Eq: D LS} with the function $\dfrac{1}{\mathbf {l}_i} \overline{v_i}$ and sum for $i = 1,2,3 $ to find 
\begin{align*}
 \sum_{i=1}^3 \frac{\lambda_{}}{\mathbf {l}_i}\int_0^\infty \! |v_i|^2 \,\mathrm dy &= - \sum_{i = 1}^3(\mathbf {l}_i)^3\int_0^\infty \! \overline {v_i} \,\partial_y^4 v_i \,\mathrm dy \\[-5pt]
&=\quad \sum_{i = 1}^3 (\mathbf {l}_i)^3\int_0^\infty \! \partial _y \overline {v_i} \,\partial_y^3 v_i \,\mathrm dy  + \overbrace{ \sum_{i = 1}^3 \overline {v_i}   \Big[(\mathbf {l}_i)^3 \partial^3_y v_i \Big] \Big|_{0}^ \infty }^{=0} \\[-5pt]
&= - \sum_{i = 1}^3 (\mathbf {l}_i)^3\int_0^\infty \! |\partial_y^2 v_i|^2 \,\mathrm dy + \overbrace{ \sum_{i = 1}^3 (\mathbf {l}_i)^2 \partial_y^2 v_i \Big[ \mathbf {l}_i \partial_y \overline{v_i} \Big]\Big|_{0}^\infty}^{= 0} \\
&= - \sum_{i = 1}^3 (\mathbf {l}_i)^3\int_0^\infty \! |\partial_y^2 v_i|^2 \,\mathrm dy \,. 
\end{align*}
Here we have used all  boundary conditions at $y = 0 $ and the fact that the functions $v_i$ and consequently  all their derivatives vanish exponentially at infinity. The latter holds  due to the fact that the solutions of the above equations are  linear combinations of exponential functions.
The facts that  $0 \neq \lambda \in \overline{\mC_+}$ and  $\mathbf {l}_i > 0 $ enforce  $v \equiv 0$.  This verifies the claim. 

Furthermore,  the matrices 
\begin{align*}
\begin{bmatrix}
  1 & 1 & 1 \\
\end{bmatrix}, \hspace{0.8cm}
\begin{bmatrix}
\mathbf {l}_1 & -\mathbf {l}_2 &  0        \\
0       & \mathbf {l}_2  & -\mathbf {l}_3
\end{bmatrix}\,,
\end{align*}
\begin{equation*}
\begin{bmatrix}
  (\mathbf {l}_1)^2 & (\mathbf {l}_2)^2  & (\mathbf {l}_3)^2 
\end{bmatrix}\,, \hspace{0.8cm}
\begin{bmatrix}
  (\mathbf {l}_1)^3 & -(\mathbf {l}_2)^3 & 0            \\
    0         & (\mathbf {l}_2)^3  & -(\mathbf {l}_3)^3
\end{bmatrix} \,
\end{equation*}
are surjective and hence the normality condition \eqref{normality condition}
is satisfied.
\subsubsection{Compatibility condition}
We next  turn our attention to the corresponding compatibility condition \eqref{compatibility}. As we have assumed  the initial planar double bubble $\Gamma^0$  fulfills the contact, angle, the curvature  and the balance of flux condition,   we see $\mu^{0} = \mathcal J  \rho^0$ and  $ \mathfrak G_j (\rho^0 ) = 0$ for $j= 1,2, \dots, 6$. This is exactly the  first condition in  \eqref{compatibility}. 

Concerning the   second equation in the compatibility condition \eqref{compatibility},  the following lemma shows that it is equivalent to  the geometric compatibility condition \eqref{Con: sum laplace curvature} if  the existence of triple junctions and the angle condition for the initial data are  already assumed. 
\begin{lem}
Under the conditions  $ \mathfrak G_j ( \rho^0 ) = 0$ $(j =1,2,3) $ and  $\mu^{0} = \mathcal J  \rho^0 $ on $\Sigma^*$,  the    second equation in the corresponding compatibility condition \eqref{compatibility} and the geometric compatibility condition \eqref{Con: sum laplace curvature} are equivalent, provided  $\rho^0$ is sufficiently  small  in the $C^1$-norm. 
\end{lem}
\begin{proof}
 The second equation in the corresponding first-order compatibility condition \eqref{compatibility} reads as \vspace{-0.2cm}
\begin{align}\label{Eq: 4 compatibility}
 \sum_{i = 1}^3  \mathfrak{F}_i( \rho_i^0 ,  \rho^0) 
+ \mathcal{B}_i ( \rho_i ,   \rho^0) \big(    \mathcal J \, \overbrace{ \big( I - \mathcal{B} ( \rho^0 ,\rho^0) \mathcal J \big) ^{-1} \mathfrak{F}( \rho^0 ,  \rho^0)}^{z:=}\,\big)_i = 0  
\end{align} 
on $\Sigma^*$. Here we have used the facts that the  zeroth-order boundary operator   $B_1 u = \sum_{i = 1}^3 u_i  $ and  $G_1 \equiv 0 $.
Let us remind that 
\begin{align*}
\mathfrak{F}_i ( \rho_i^0 ,  \rho^0)&= \frac{1}{\langle n_i^*, n^0_i \rangle} \Delta \big( .\,  , \rho^0_i , ( \mathcal J \rho^0)_i \big) \kappa_i \big(.\,, \rho_i^0 , ( \mathcal J \rho^0 )_i^{} \big),\quad\mathfrak{B}_i ( \rho^0_i ,   \rho^0)&= \frac{\big \langle n^{}_{\partial \Gamma_i^*} , n_i^0 \big \rangle}{\big \langle n_i^* , n_i^0 \big \rangle},
\end{align*}
and   $\tau_i^* = n^{}_{\partial \Gamma_i^*}$  on $\Sigma^*$. 

On the other hand, the angle condition implies
\begin{align*}
 \big \langle n_i^* , n_i^0 \big \rangle  =\big \langle n_j^* , n_j^0 \big \rangle\,,  \quad \big \langle n^{}_{\partial \Gamma_i^*} , n_i^0  \big \rangle =
  \big \langle n^{}_{\partial \Gamma_j^*} , n_j^0  \big \rangle \qquad \text{on } \Sigma^* .  
\end{align*}
Thus \eqref{Eq: 4 compatibility} can be rewritten as 
$$
\frac{1}{\langle n_1^*, n^0_1 \rangle} \sum_{i =1}^3 \Delta \big( .\,  , \rho^0_i , ( \mathcal J \rho^0)_i \big) \kappa_i \big(.\,, \rho_i^0 , ( \mathcal J \rho^0 )_i^{} \big)+ \mathfrak B_1 \sum_{i = 1}^3 (\mathcal J z )_i  = 0  \quad \text{ on } \Sigma^*   \,, 
$$
where $\langle n_1^*, n^0_1 \rangle \neq 0 $ if $\Gamma^0$ is close enough to $\Gamma^*$ in $C^1$-norm, that is if $\rho^0$ is sufficiently  small  in the $C^1$-norm.

Moreover, due to the definition of the matrix $\mathcal J$, we have 
$$
\sum_{i = 1}^3 \big( \mathcal J y )_i = 0  \quad \forall y \in \mR^3  \,. $$
Hence the compatibility condition  \eqref{Eq: 4 compatibility} is equivalent to 
$$
\sum_{i = 1}^3 \Delta \big( \sigma, \rho^0_i , ( \mathcal J \rho^0)_i \big) \kappa_i \big(\sigma, \rho_i^0 , ( \mathcal J \rho^0 )_i^{}\big)= 0 \,,
$$
which is exactly the geometric compatibility condition \eqref{Con: sum laplace curvature} written in a parameterization. This finishes the proof.
\end{proof}
\subsection{The spectrum of the linearized problem} \label{Sec. DB spectrum}
Since $\Omega = [-l_1^*, l_1^*] \subset \mR$, the linearized operator $A_0$ (see   \eqref{linear operator}) is defined as $A_0 u = A u $ with domain
$$
D(A_0)=\Big\{ u\in C^{4}(\o): Bu=0 \text{ on } \partial \Omega\Big\},
$$
where $A$ and $B$ is defined in Section \ref{Sec: general setting}. 
Due to Remark 2.2 in \cite{AbelsArabGarcke}, the spectrum of the linearized operator $A_0$ consists entirely of eigenvalues. As the analysis of the eigenvalue problem is invariant under  the change of variables, we switch to the setting where the functions $u_i$ ($i = 1,2,3$) have different domains.

Now, the eigenvalue problem for the linearized operator $A_0$  reads as follows:
For $i = 1,2,3$,
\begin{equation} \label{Eq: D eigenvalue linear}
   \Delta_{\Gamma_i^*} \big(  \Delta_{\Gamma_i^*} u_i + (\kappa ^*_i)^2 u_i \big) = \lambda u_i \quad \text{ in } \Gamma_i^* \quad (i=1,2,3)\,, 
\end{equation} 
subject to the  boundary conditions on $\Sigma^*$
\begin{equation} \label{Eq: D eigenvalue linear BCs 1}
\left\{
 \begin{aligned}
 u_1+ u_2  + u_3 &=0 \,,  \\
   q_i^* u_i + \partial_{n_ {\partial\Gamma_i^*}} u_i &= q_j^* u_j + \partial_{n_ {\partial\Gamma_j^*}} u_j \,,  \\
\textstyle{\sum}_{i=1}^3\Delta_{\Gamma_i^*} u_i + (\kappa ^*_i)^2 u_i &= 0 \,, \\ 
   \partial_{n_ {\partial\Gamma_i^*}} \big( \Delta_{\Gamma_i^*} u_1 + (\kappa_i^{*})^2 u_i  \big)                     &=\partial_{n_ {\partial\Gamma_j^*}} \big( \Delta_{\Gamma_j^*} u_j + (\kappa_j^{*})^2 u _j  \big) \,,                
 \end{aligned}
\right.
\end{equation}
where $(i,j) = (1,2), (2,3)$.

To   derive a bilinear form associated with this eigenvalue problem, let us multiply the equation \eqref{Eq: D eigenvalue linear} by $-\big(\Delta_{\Gamma_i^*} \overline{u_i} + (\kappa ^*_i)^2 \overline{u_i}\big)$ and then integrate by parts and sum over $i = 1, 2, 3$  to find 
\begin{align*}
 \sum_{i =1}^3  \int_{\Gamma_i^*}& \big|\nabla_{\Gamma_i^*} \big(  \Delta_{\Gamma_i^*} u_i + (\kappa ^*_i)^2 u_i \big) \big|^2  \mathrm d s
& = -\lambda \sum_{i =1}^3 \int_{\Gamma_i^*} u_i \big( \Delta_{\Gamma_i^*} \overline{u_i} + (\kappa ^*_i)^2 \overline{u_i} \big)  \,\mathrm d s \,.
\end{align*}
Here, as usual,   we have used the last two boundary conditions.
We observe further
\begin{align*}
-\sum_{i =1}^3 \int_{\Gamma_i^*} u_i \big( \Delta_{\Gamma_i^*} \overline{u_i} + (\kappa ^*_i)^2 \overline{u_i} \big)  \,\mathrm  \mathrm d s= \sum_{i =1}^3 \int_{\Gamma_i^*} \big|\nabla_{\Gamma_i^*}u_i \big|^2 &- (\kappa_i^*)^2 |u_i|^2 \,\, \mathrm  \, \mathrm d s \\
&{}-  \sum_{i =1}^3 \int_{\Sigma^*} u_i \, \partial_{n_ {\partial \Gamma_i^*}}u_i\overline{u_i} \,.
\end{align*}
On the other hand 
\begin{align*}
  \sum_{i =1}^3 \int_{\Sigma^*} u_i \, \partial_{n_ {\partial \Gamma_i^*}}\overline{u_i} &=   \sum_{i =1}^3 \int_{\Sigma^*} \big( u_i \, \partial_{n_ {\partial \Gamma_i^*}}\overline{u_i}+ q^*_i |u_i|^2 - q_i^* |u_i|^2 \big) \\
&=    \sum_{i =1}^3 \int_{\Sigma^*} \big( \partial_{n_ {\partial \Gamma_i^*}}u_i + q_i^* u_i \big)\overline{u_i} - \sum_{i =1}^3 \int_{\Sigma^*} q_i^* |u_i|^2 \\
&= \int_{\Sigma^*} \big( \partial_{n_ {\partial \Gamma_1^*}}u_1 + q_1^* u_1 \big)\underbrace{ \sum_{i =1}^3 \overline{u_i}}_{=0} - \sum_{i =1}^3 \int_{\Sigma^*} q_i^* |u_i|^2 \\[-11pt]
& = - \sum_{i =1}^3 \int_{\Sigma^*} q_i^* |u_i|^2 \,.
\end{align*}
We now  combine the three equalities above to discover
\begin{align}\label{Eq: D bilinear}
 \sum_{i =1}^3  \int_{\Gamma_i^*} \big|\nabla_{\Gamma_i^*} \big(  \Delta_{\Gamma_i^*} u_i + (\kappa ^*_i)^2 u_i \big) \big|^2  \mathrm d s
&= \lambda I(u,u) \,,
\end{align}
where
\begin{align} \label{def: D second variation}
I(u,u) &:= \sum_{i =1}^3 \int_{\Gamma_i^*} \big|\nabla_{\Gamma_i^*}u_i \big|^2 - (\kappa_i^*)^2 |u_i|^2 \,\, \mathrm  \, \mathrm d s + \sum_{i =1}^3 \int_{\Sigma^*} q^*_i |u_i|^2 \nonumber \\ 
       & \; {}=-  \sum_{i =1}^3 \int_{\Gamma_i^*} \overline{u_i} \big( \Delta_{\Gamma_i^*} u_i + (\kappa ^*_i)^2 u_i \big)  \,\mathrm d s +     \sum_{i =1}^3 \int_{\Sigma^*}  \big( \partial_{n_ {\partial \Gamma_i^*}}u_i + q_i ^* u_i \big) \overline{u_i}\,. \end{align}

Note carefully that   in \eqref{def: D second variation} we just used integration by parts to obtain the second equality. It is interesting now to see   that although (due to the linearized angle condition and the fact that on the boundary $u_1 + u_2 + u_3 = 0 $) we have  
\begin{equation} \label{Eq: DB weak angle}
\sum_{i =1}^3 \int_{\Sigma^*}  \big( \partial_{n_ {\partial \Gamma_i^*}}u_i + q_i^* u_i \big) \overline{u_i}  = 0 \,,
\end{equation}
but nevertheless  this does not effect the value of $I (u,u)$ (cf. \cite[Remark 3.7]{HutchingsMorganRitorAntonio}). 
\begin{rem}\label{Rem: real}
The identity \eqref{Eq: D bilinear} in particular shows that $\lambda \in \mR$.
\end{rem} 
\begin{rem}
 Indeed as one may have expected, the linearized problem \eqref{Eq: D linear 1}, \eqref{Eq: D linear BCs 1} is the gradient flow of  the energy functional 
$$
E(u)=  \frac{I(u,u)}{2} \,,
$$
with respect to the $H^{-1}$-inner product, see for instance \cite{GarckeItoKohsakatriplejunction}.
\end{rem}

\subsubsection{Related problem: Double bubble conjecture}

The goal of this section is to prove that, a part from  zero, the spectrum of the linearized problem lies in $\mR_{+}$. We do this by considering  the bilinear form $I ( \, , \, )$.

  In the following  we state the  second variation formula    proved in general dimension by Morgan and co-authors:
\begin{pr} \label{Prop: second variation formula} 
\emph{(\cite[Proposition 3.3]{HutchingsMorganRitorAntonio}).}
Let $\Gamma^*$ be a stationary planar double bubble and let  $\varphi_t $ be a  one-parameter variation which preserves the  areas of enclosed regions. Furthermore   denote by $L(t)$ the length of $\varphi_t( \Gamma^*)$. Then   \begin{align*}
\frac{d^2}{dt^2} L(t)\big|_{t = 0} = I(u,u) \,,
\end{align*}
where $u_i = \langle \frac{d}{dt} \varphi_t , n_i^* \rangle $. 
\end{pr}
Here and hereafter, by (one-parameter) variations $    \{\varphi_t \}_{|t| < \epsilon}:\Gamma \to \mR^2$ of a double bubble $\Gamma \subset \mR^2$ we mean   the variations which are smooth   (up to the  triple junctions)  having equal values along triple junctions.\begin{rem} Notice   that   in \eqref{def: D second variation} we have used  outer unit conormals where inner unit conormals are used in \cite{HutchingsMorganRitorAntonio}. In addition, the constants $q^*_i$ and their corresponding ones in \cite{HutchingsMorganRitorAntonio} are also opposite in signs due to the different choice of normals. This  explains the sign differences.
\end{rem}
\begin{rem}
Of course,  a double bubble is stationary for any variation preserving the area of the enclosed regions if and only if it is stationary for the surface diffusion flow \eqref{Eq: 2-double bubble}, see Section \ref{Sec: Equilibria} and \cite[page 465]{HutchingsMorganRitorAntonio}.
\end{rem}
Following \cite{HutchingsMorganRitorAntonio}, we denote by  $ \mathcal F(\Gamma)$  the space of functions $u \in H^1( \Gamma ) $ satisfying  
\begin{equation*}
\left \{ 
\begin{aligned}
u_1 + u_2 + u_3 = 0 \quad \text{ on } \Sigma \,,\\
\int_{\Gamma_1} \!u_1 = \int_{\Gamma_2} \! u_2 = \int_{\Gamma_3}  \!u_3 \,.
\end{aligned}
\right.
\end{equation*}
\begin{lem} \label{Lem: variation}
\emph{(\cite[Lemma 3.2]{HutchingsMorganRitorAntonio}).} Let $\Gamma^*$ be a stationary double bubble. Then for any smooth $u \in \mathcal F (\Gamma^*)$ there is an area preserving variation $\{ \varphi_t\}$ of $\Gamma^*$  such that  the normal components of the associated infinitesimal vector field are the functions $u_i$, i.e., $u_i = \langle \frac{d}{dt} \varphi_t , n_i^* \rangle$,  $i = 1, 2, 3$. 
\end{lem}
We are now ready to present:
\begin{df}[The concept of stability in differential geometry]
A double bubble $\Gamma^*$ is said to be variationally stable if it is stationary and 
$$
I (u,u) \geq 0 \qquad \forall \, u \in \mathcal F ( \Gamma^*)\,.
$$
\end{df}
We are forced  here to name the concept of stability in differential geometry  \textit{variationally stable} instead of \textit{stable}. Indeed it is an open problem whether for  double bubbles this concept of stability in differential geometry is equivalent to the concept of stability in  PDE theory.
There are several  evidences in this work which show how closely these two concepts are, starting from Lemma \ref{Lem: spectrum negative} below.
\begin{rem}
Note that the   concept of stability in differential geometry was called  \textit{stable} in \cite{HutchingsMorganRitorAntonio}.
\end{rem}

\begin{cor}
A  perimeter-minimizing double bubble for prescribed areas is variationally stable.
\end{cor}
\begin{proof}
Let $\Gamma$ be a  primeter-minimizing double bubble. As a minimizer, the second derivative of length is nonnegative along all variations which preserve the area, in other words by Proposition \ref{Prop: second variation formula} $I ( u, u) \geq 0$ for all functions $u$ given by  normal components of   area preserving variations. On the other hand, by Lemma \ref{Lem: variation} we know that  every  smooth element of $\mathcal F ( \Gamma)$ is of this form. Therefore  $I ( u, u) \geq 0 $ for all $u \in \mathcal F ( \Gamma)$, which  finishes the proof.
\end{proof}
\begin{theo}
\emph{(\cite[Theorem 2.9]{FoisyAlfaroBrock...}).}
The standard planar double bubble is the unique perimeter-minimizing double bubble enclosing and separating two given regions of prescribed areas.
\end{theo} 
Therefore as  an important corollary one gets: (see also \cite[Theorem 3.2]{MorganWichiramala}) \begin{cor} \label{Cor: D stable}
The standard planar double bubble is variationally stable.
\end{cor}
We are now ready to see the first evidence.

\begin{lem} \label{Lem: spectrum negative}
$\sigma (A_0) \setminus \{ 0 \} \subset \mR_+$. 
\end{lem}
\begin{proof}
 Let $\lambda \in \sigma(A_0) \setminus \{ 0 \} $. As mentioned before the spectrum consists entirely of eigenvalues. In addition, according to Remark \ref{Rem: real}, $\lambda$ is real. 

Therefore, let  $\lambda $ be an eigenvalue with a corresponding eigenvector $u \in C^{4 + \a}( \Gamma^{*})$. This means $u$ solves  the eigenvalue problem \eqref{Eq: D eigenvalue linear} subject to the boundary conditions \eqref{Eq: D eigenvalue linear BCs 1} for $\lambda$. Since $ \lambda \neq 0$, we  deduce after integrating \eqref{Eq: D eigenvalue linear}:
\begin{equation*}
\int_{\Gamma^*_1} \! u_1 = \int_{\Gamma_2^*}  \! u_2 = \int_{\Gamma_3^*} \! u_3 \,,
\end{equation*}
 where we employed the divergence theorem and the last boundary condition. This together with the first boundary condition implies that  $u \in \mathcal F ( \Gamma^*)$. Therefore  $I ( u, u) \geq 0$ by Corollary \ref{Cor: D stable}. 

Now assume   $I ( u, u) = 0$. In view of the equation \eqref{Eq: D bilinear}, we obtain 
$$
\Delta_{\Gamma_i^*} u_i + (\kappa ^*_i)^2 u_i = c_i  
$$
for  some constants $c_i$ (cf. \cite[Lemma 3.8]{HutchingsMorganRitorAntonio}). This together with the equation \eqref{Eq: D eigenvalue linear} immediately implies $ u \in N(A_0)$, i.e., $\lambda = 0$, a contradiction. Thus $I ( u,u) > 0$ for the eigenvector $u$. Now $\lambda > 0$ by  \eqref{Eq: D eigenvalue linear}. 
This finishes the proof.
\end{proof}
The bilinear form $I (\,,\,)$ is further discussed in Appendix \ref{Sec: more about bilinear}.
\subsection{Null space of the linearized problem} \label{Section null}
We next determine  the null space of the linearized operator $A_0$. That is, we consider the case  $\lambda = 0$ in the eigenvalue problem \eqref{Eq: D eigenvalue linear},\eqref{Eq: D eigenvalue linear BCs 1}. 

Using the identity  \eqref{Eq: D bilinear}, we  easily get $u \in N(A_0) $ if and only if there exists a constant vector $ c = (c_1, c_2, c_3) \in \mR^3$ such that
\begin{equation} \label{Eq: linear ci}
\Delta_{\Gamma_i^*} u_i + (\kappa ^*_i)^2 u_i = c_i   \quad \text{ on } \Gamma_i^* \quad (i = 1,2,3), 
\end{equation} 
subject to the   conditions 
\begin{equation} \label{Eq: linear B ci}
\left \{
\begin{aligned}
u_1+ u_2  + u_3 &=0 & & \text{ on } \Sigma^* , \\
   q_1^* u_1 + \partial_{n_ {\partial\Gamma_1^*}} u_1 &= q_2^* u_2 + \partial_{n_ {\partial\Gamma_2^*}} u_2 = q_3^* u_3 + \partial_{n_ {\partial\Gamma_3^*}} u_3 && \text{ on } \Sigma^* , \\
c_1 + c_2 + c_3 &= 0 \,. 
\end{aligned}
\right.
\end{equation}
Notice that the constant vector  $ c =  c (u)$ depends linearly on $u$ by  \eqref{Eq: linear ci}.
\begin{df}
Following \cite{HutchingsMorganRitorAntonio}, we define the space of Jacobi functions 
$$
\mathcal J(\Gamma^*) := \big \{ u \in N(A_0) :  c = c (u) = 0 \, \big\}\,.
$$
\end{df}
 
 We need, for  later use, an identity that relates the null space $N(A_0)$ to the bilinear form $I(\,,\,)$.

\begin{lem} \label{Lem: null bilinear}
Assume $u \in N(A_0)$. Then
$$
I(u,u) =-\sum_{i=1}^3  c_i \int_{\Gamma_i^*} \!u_i\,,
$$ 
where the constants $c_i$, satisfying $\sum_{i=1}^3 c_i = 0$,  depend linearly on $u$ by  \eqref{Eq: linear ci}.  
\end{lem}
\begin{proof}
By inserting  \eqref{Eq: linear ci} into the  definition of  the bilinear form \eqref{def: D second variation} and taking into account the equation \eqref{Eq: DB weak angle} coming from the first two  boundary conditions in  \eqref{Eq: linear B ci}, we get the desired identity.  
\end{proof} 
As a  corollary we get
\begin{cor} \label{Cor: null bilinear}
If $u \in N(A_0) \, \cap \mathcal F( \Gamma^*)$, then $I(u,u) = 0$.
\end{cor}
Let us  rewrite the linear equations \eqref{Eq: linear ci} as a system of linear nonhomogeneous second order ordinary differential equations with constant coefficients   
$$
\partial_x^2 u_i + (\kappa ^*_i)^2 u_i =c_{i} \quad \text{ for } x \in [-l_i^{*} , l^{*}_i] \quad (i = 1,2,3)\,,
$$
with the  conditions
\begin{equation*}
\left \{
\begin{aligned}
u_1+ u_2  + u_3 &=0 \ & & \text{ on } \Sigma^* , \\
   q_1^* u_1 \pm \partial_x u_1 &= q_2^* u_2 \pm \partial_x u_2 = q_3^* u_3 \pm  \partial_{x} u_3 & & \text{ on } \Sigma^* , \\
c_1 + c_2 + c_3 &= 0 \, 
\end{aligned}
\right.
\end{equation*}
for the functions  $ u_i:[-l^*_i, l^*_i]  \rightarrow \mR $.
\subsubsection{Determination of   Jacobi functions
}

Let us first consider the case $\kappa^*_2 \neq 0$. The general solution of the  linearized problem is then
\begin{equation} \label{Eq: DB general solution}
u_i (x) = a_i \sin (\kappa_i^* x) + b_i \cos (\kappa_i^* x)  \quad (i=1,2,3)\,.
\end{equation}
We   calculate at  $x = \pm l^*_1$
\begin{align*}
q_1^* u_1 &= \mp \cot ( \gamma^* + \tfrac{\pi}{3}) \kappa_1^* a_1 \sin(\gamma^* + \tfrac{\pi}{3} )  +\cot ( \gamma^* + \tfrac{\pi}{3}) \kappa_1^* b_1 \cos( \gamma^* + \tfrac{ \pi} {3})  \\
&=\mp  a_1 \kappa_1^* \cos ( \gamma^* +  \tfrac{\pi}{3})+ b_1 \kappa_1^* \cot ( \gamma^* + \tfrac{\pi}{3}) \cos( \gamma^* + \tfrac{ \pi} {3})  \,,\\[6pt]
\pm \partial_x u_1 &= \pm a_1 \kappa_1^* \cos ( \gamma^* +  \tfrac{\pi}{3})+ b_1 \kappa_1^* \sin( \gamma^* + \tfrac{\pi}{3})\,. 
\end{align*}
Therefore  
$$
q_1^* u_1 \pm \partial_x u_1 = b_1 \frac{\kappa_1^*}{\sin( \gamma^* + \tfrac{\pi}{3})}
 \quad \text{ at } x = \pm l^*_1 \,.
$$
Similarly we get
\begin{align*}
q_2^* u_2 \pm \partial_x u_2 = b_2 \frac{\kappa_2^*}{\sin( \gamma^* - \tfrac{\pi}{3})}
 \quad \text{ at } x = \pm l^*_2 \,,\\
q_3^* u_3 \pm \partial_x u_3 = b_3 \frac{\kappa_3^*}{\sin( \gamma^* - \pi)}
 \quad \text{ at } x = \pm l^*_3 \,.
\end{align*}
Thus we conclude 
\begin{align*}
 b_1 \frac{\kappa_1^*}{\sin( \gamma^* + \tfrac{\pi}{3})}
= b_2 \frac{\kappa_2^*}{\sin( \gamma^* - \tfrac{\pi}{3})} = b_3 \frac{\kappa_3^*}{\sin( \gamma^* -\pi)}
\,.
\end{align*}
Furthermore, $u_1 (\pm l^*_1) + u_2 (\pm l^*_2) + u_3 (\pm l^*_3) = 0 $ reads as 
\begin{align*}
\mp a_1 \sin( \gamma^* + \frac{\pi}{3}) &+ b_1 \cos ( \gamma^* + \frac{\pi}{3})  \\
&{}\mp a_2 \sin( \gamma^* - \frac{\pi}{3}) + b_2 \cos ( \gamma^* - \frac{\pi}{3})  \\
&\hspace{2.7cm}\mp a_3 \sin( \gamma^* - \pi) + b_3 \cos ( \gamma^* - \pi) 
=0 \,. 
\end{align*}
Altogether, we have to find solutions to the following system
\begin{equation*}
\left\{
\begin{aligned}
&a_1 \sin( \gamma^* + \frac{\pi}{3}) + a_2 \sin( \gamma^* - \frac{\pi}{3})+  a_3 \sin( \gamma^* - \pi) = 0  \,, \\
&b_1 \cos ( \gamma^* + \frac{\pi}{3}) + 
b_2 \cos ( \gamma^* - \frac{\pi}{3}) +
b_3 \cos ( \gamma^* - \pi)  = 0 \,,\\
&b_1 \frac{\kappa_1^*}{\sin( \gamma^* + \tfrac{\pi}{3})}
 = b_2 \frac{\kappa_2^*}{\sin( \gamma^* - \tfrac{\pi}{3})} = b_3 \frac{\kappa_3^*}{\sin( \gamma^* -\pi)} \,.
\end{aligned}
\right.
\end{equation*}
 Due to the identities \eqref{Eq: D useful one} and  \eqref{cos and sin} we  get 
\begin{equation*}
\left\{ 
\begin{aligned}
&a_1 \sin( \gamma^* + \frac{\pi}{3}) + a_2 \sin( \gamma^* - \frac{\pi}{3})+  a_3 \sin( \gamma^* - \pi) = 0 \,, \\
&b_1 = b_2 = b_3 \,.  \\
\end{aligned}
\right.
\end{equation*}
Therefore, in view of the formula \eqref{cos and sin}, we obtain 
$$
(a_1, a_2, a_3) \in \mathrm{span}\Big\{ (1,1,1), \big(0, \,-\,  \tfrac{ \sin(\gamma^* - \pi)}{ \sin(\gamma^*-  \frac{\pi}{3})},\, 1\big) \Big \}, \quad (b_1, b_2, b_3 ) \in \mathrm{span} \{(1,1,1) \}\,.
$$
This shows the following lemma:
\begin{lem} \label{Lem: Jacobi vector}
Assume $\kappa_2^* \neq 0$. Then the space of Jacobi functions is a three dimensional  vector space whose  basis  consists of  
\begin{equation*}
v^{(1)} =
 \begin{pmatrix}
    \cos (\kappa_1^* x )\\[4pt]
    \cos (\kappa_2^* x )\\[4pt]
    \cos (\kappa_3^* x )
  \end{pmatrix}, \quad
v^{(2)} 
   =\begin{pmatrix}
    \sin (\kappa_1^* x )\\[4pt]
    \sin (\kappa_2^* x )\\[4pt]
    \sin (\kappa_3^* x )
  \end{pmatrix}, \quad
v^{(3)} 
  =\begin{pmatrix}
    0\\[4pt]
     \frac{ \sin(\gamma^*)}{ \sin(\gamma^*-  \frac{\pi}{3})}\sin (\kappa_2^* x )\\[4pt]
\sin (\kappa_3^* x )
  \end{pmatrix}\,.
\end{equation*} 
\end{lem}
We now consider the case $\kappa_2^* = 0$. The general solution of the  linearized problem is then
\begin{align*}
u_1 &= a_1 \sin (\kappa_1^* x) + b_1 \cos (\kappa_1^* x)\,,\quad  u_2 = a_2 x + b_2 \,,\\
u_3 &= a_3 \sin (\kappa_3^* x) + b_3 \cos ( \kappa_3^* x)\big(= - a_3 \sin (\kappa_1^* x) + b_3 \cos ( \kappa_1^* x) \big) \,,
\end{align*}
where we used the fact that  $ \kappa_3^* = - \kappa^*_1$ in case $\gamma^* = \frac{\pi}{3}$. 
Let us  also remind that for $\gamma^* = \frac{\pi}{3}$ we have
$$
q^*_2 = \frac{\kappa_1^*}{\sin(\frac{\pi}{3})} \quad \text{ and } \quad  l_2^* = -\frac{\sin(\frac{\pi}{3})}{\kappa^*_1}
$$ and  so $q^*_2 \, l^*_2 = -1$.
Therefore, 
\begin{align*}
q^*_2 u_2 \pm \partial_x u_2 &= \mp a_2 + \frac{\kappa^*_1}{\sin(\frac{\pi}{3})}b_2 \pm a_2 =  \frac{\kappa_1^*}{\sin(\frac{\pi}{3})}b_2 & & \text{at } x= \pm l^*_2 \,. 
\end{align*}
Taking into account the calculation done previously for $u_1$ and $u_3$,
 the condition  $q_1^* u_1 \pm \partial_x u_1 = q_2^* u_2 \pm \partial_x u_2 = q_3^* u_3 \pm  \partial_{x} u_3$ reads as
$$
b_1 \frac{\kappa_1^*}{\sin( \tfrac{\pi}{3})}
= b_2\frac{\kappa_1^*}{\sin( \tfrac{\pi}{3})}= b_3 \frac{\kappa_3^*}{\sin(- \frac{2 \pi}{3})} \Big(= b_3 \frac{-\kappa_1^*}{-\sin( \tfrac{\pi}{3})} =b_3 \frac{\kappa_1^*}{\sin( \tfrac{\pi}{3})}
 \Big)\,.
$$
Therefore, we conclude $b_1 = b_2 = b_3$.
Furthermore, $u_1 (\pm l^*_1) + u_2 (\pm l^*_2) + u_3 (\pm l^*_3) = 0 $ reads as 
\begin{align*}
\mp a_1 \sin( \tfrac{\pi}{3}) &+ b_1 \cos ( \tfrac{2\pi}{3})\mp a_2\frac{\sin(\frac{\pi}{3})}{\kappa^*_1}+ b_2 \pm a_3 \sin( \tfrac{\pi}{3}) + b_3 \cos ( \tfrac{2\pi}{3}) =0 \,. 
\end{align*}
Moreover, using the facts that $b_1 = b_2 = b_3$ and $\cos (\frac{2 \pi}{3}) = -\frac{1}{2}$, we see
that$$
b_1 \cos ( \tfrac{2\pi}{3}) + b_2 + b_3 \cos ( \tfrac{2\pi}{3})=0 \,. 
$$
In summary, we have to find solutions to the following system
\begin{equation*}
\left\{
\begin{aligned}
&a_1  + \frac{a_2}{\kappa^*_1}-a_3  = 0  \,, \\
&b_1  = b_2= b_3.
\end{aligned}
\right.
\end{equation*}
Therefore,
$$
(a_1, a_2, a_3) \in \mathrm{span}\big\{ (1, 0 , 1), (0, \kappa_1^*, 1) \big \}, \quad (b_1, b_2, b_3 ) \in \mathrm{span} \{(1,1,1) \}\,.
$$

\begin{lem}
Assume $\kappa_2^* = 0$. Then the space of Jacobi functions is $3$-dimensional and its basis is given by  
\begin{equation*}
v^{(1)} =
 \begin{pmatrix}
    \cos (\kappa_1^* x )\\[4pt]
    1                   \\[4pt]
    \cos (\kappa_1^* x )
  \end{pmatrix}, \quad
v^{(2)} 
   =\begin{pmatrix}
    \sin (\kappa_1^* x )\\[4pt]
    0                   \\[4pt]
    -\sin (\kappa_1^* x )
  \end{pmatrix}, \quad
v^{(3)} 
  =\begin{pmatrix}
    0\\[4pt]
    \kappa_1^* x\\[4pt]
-\sin (\kappa_1^* x )
  \end{pmatrix}.
\end{equation*} 
\end{lem}
\subsubsection{ The null space $N(A_0)$ is at most five-dimensional}
Next we try to get an upper bound on the dimension of the null space.
\begin{lem} \label{Lem: DB null < 5}
The null space  $N(A_0)$ of the linearized operator $A_0$ is at most five-dimensional.
\end{lem}
\begin{proof}
We have already shown that the space of Jacobi functions is three-dimensional. Therefore it is enough to show that there exist at most  two independent vectors in the null space $N(A_0)$ for which $ c \neq 0$. 

Take any three vector functions  $u^{(1)}, u^{(2)}, u^{(3)}\in N(A_0)$ for which the vector constants $ c^{(i)} =  {c^{(i)}} (u^{(i)}) \neq 0$, $i = 1,2,3$. Then as 
$$
 c^{(i)}  \in \big\{ \,  c = (c_1,c_2,c_3)\in \mR^3 : c_1+c_2+ c_3 = 0 \big \} 
$$ which is a two dimensional subspace of $\mR^3$, there exist scalars $a_1,a_2,a_3$, not all zero, such that
$$
0 = \sum_{j= 1}^3 a_i  {c^{(i)}} = \sum_i^3 a_i T u^{(i)}= T (\sum_{i=1}^3 a_iu^{(i)}) \,. 
$$  
Here $T$ is the linear  operator defined by the left hand side of \eqref{Eq: linear ci}. Thus we get $\sum_{i=1}^3 a_iu^{(i)} \in \mathcal J(\Gamma^*)$, in other words,
$$
 \sum_{i=1}^3 a_iu^{(i)} = \sum_{j=1}^3 b_j v^{(j)},
$$ 
where $\{ v^{(1)}, v^{(2)}, v^{(3)} \}$ is a basis of  $J(\Gamma^*)$. This means that the vectors 
$$
u^{(1)}, u^{(2)}, u^{(3)}, v^{(1)}, v^{(2)}, v^{(3)}
$$
are linearly dependent, and this completes the proof.
\end{proof} 
 Indeed, we will prove in Corollary \ref{cor: null and manifold} below that the dimension of the null space is exactly five.
\subsection{Manifold of equilibria} \label{Sec: D manifold} 
Our goal  in this section is to prove that near $\rho\equiv 0$, which corresponds to $\Gamma^*$, the set $\mathcal E$ of equilibria  of  the nonlinear system \eqref{Eq: D nonlinear semifinal}   creates a smooth  manifold of dimension $5$.
\subsubsection{Equilibria  of  the nonlinear system
} 
Let us first identify the set of equilibria $\mathcal E$ of  the nonlinear system \eqref{Eq: D nonlinear semifinal}. According to \eqref{set of equilibria E},  $\rho \in \mathcal E$ if and only if  for $i = 1, 2, 3 $ and $j = 1, 2, . . . , 6$, 
\begin{equation*}  
     \left \{
   \begin{aligned}
   \rho &\in B_{X_1}(0,R)\,, \\
     0 &= \mathfrak{F}_i( \rho_i ,  \rho|_{\Sigma^*})   \\&  \quad + \mathfrak{B}_i ( \rho_i ,   \rho|_{\Sigma^*}) \Big(  \Big \{ \mathcal J  \big( I - \mathfrak{B} ( \rho ,\rho|_{\Sigma^*}) \mathcal J \big)^{-1} \mathfrak{F}( \rho ,  \rho|_{\Sigma^*})  \Big \} \circ \mathrm{pr}_i \Big)_i  & &\text{ on } \Gamma^*_i \,, \\
 0 &= \mathfrak G_j ( \rho )   & &\text{ on } \Sigma^*. \\
   \end{aligned}
\right.
\end{equation*}
 Similarly as done in Section \ref{Sec: DB nonlocal/linear}, we can write the first three equations as a vector identity on $\Sigma^*$ and thereby obtain $\mathfrak{F}( \rho ,  \rho|_{\Sigma^*}) = 0$. Thus
\begin{equation*}
   \rho \in \mathcal E \Leftrightarrow
   \left \{
   \begin{aligned}
 & \rho   \in B_{X_1}(0,R)\,,\\
      & 0 = \mathfrak{F}_i( \rho_i ,  \rho|_{\Sigma^*})& &\text {on } \Gamma_i^* \,, & &i = 1,2,3 \,, \\
      & 0 = \mathfrak G_j ( \rho )   & &\text{on } \Sigma^* \,, & &j = 1,\dots,6 \,.\\
      \end{aligned}
\right.
\end{equation*}
Taking into account \eqref{b_i}, the definition of $\mathfrak F_i$, the balance of flux conditions $\mathfrak G_5, \mathfrak G_6 $ and the condition on curvature $\mathfrak G_4$, by applying the Gauss
theorem, we see 
\begin{equation*}
   \rho \in \mathcal E \Leftrightarrow
   \left \{
   \begin{aligned}
      &\rho \in B_{X_1}(0,R)\,, \\
      & \kappa_{i} \big( \rho_i, (\mathcal{J}\rho \circ \mathrm{pr})_i \big) \text{ are constant, }  & &\text {on } \Gamma^* \,, \\
      & \mathfrak G_j ( \rho ) = 0  & &\text{on } \Sigma^* \,, & &j = 1,2,3,4\,.\\
   \end{aligned}
\right.
\end{equation*}
Therefore, using Lemma \ref{Con:contact}, we conclude: 
\begin{align*}
 \mathcal E = \Big\{ \,\rho &\in B_{X_1}(0,R) : \rho \text{ parameterizes a standard planar double bubble} \, \Big\}.
\end{align*}

\subsubsection{Level set  representation of standard  double bubbles}
Next we represent   standard planar double bubbles  as  a subset of  the zero level sets of some smooth functions. Let  $S_{r_i}(O_i)$, $i= 1,2,3$, be the corresponding circles to standard planar double bubble $\Gamma = \{ \, \Gamma_1, \Gamma_2, \Gamma_3 \, \}$. In other words, $\Gamma_i \subset S_{r_i}(O_i)$, where $r_i$ and $O_i$ are the radius and  the center of $\Gamma_i$ respectively. 
\begin{lem}\label{Rep: SPDB} 
Let $\Gamma = DB_{r, \gamma, 0}(0, 0)$. Then 
$$
 \big \{ \, \sigma \in \mR^2 : G_i(\sigma, r, \gamma) = 0 \, \big \}= S_{r_i}(O_i)\supset \Gamma_i \qquad (i=1,2,3),
$$
where 
$ 
G_i: \mR^2\times (0, \infty) \times (0, \tfrac{2 \pi}{3}) \to \mR 
$
are smooth functions defined by   
\begin{align*}
r\sin ( \gamma + \tfrac{ \pi}{3})G_1(\sigma, r, \gamma) &= \tfrac{1}{2}   \sin(\gamma + \tfrac{\pi}{3}) \Big( |\sigma|^2 - r^2 \Big)\,, \\
r\sin ( \gamma + \tfrac{ \pi}{3}) G_2(\sigma, r, \gamma) &= \tfrac{1}{2} \Big( \sin( \gamma - \tfrac{\pi}{3})|\sigma|^2 -  2 r\sin( \tfrac{\pi}{3}) \big\langle \sigma, (1, 0) \big \rangle - r^2\sin (\gamma- \pi)  \Big)\,,\\ r\sin ( \gamma + \tfrac{ \pi}{3})
 G_3(\sigma , r,  \gamma) &= \tfrac{1}{2} \Big(\sin (\gamma - \pi) |\sigma|^2+ 2 r\sin( \tfrac{\pi}{3}) \big \langle \sigma, (1, 0) \big \rangle -r^{2} \sin( \gamma - \tfrac{\pi}{3})   \Big)\,,  
\end{align*}
 with the property that 
\begin{equation} \label{Id: sum to zero}
G_1 + G_2+ G_3 = 0 \,. 
\end{equation}
\end{lem}
The proof is given in Appendix \ref{App. Rep: SPDB}. Next let us  look at the gradient of $G_i$. 
\begin{lem} \label{Cor: gradient G}
Let  $\Gamma = DB_{r, \gamma, 0}(0, 0)$. Then
$$
\nabla_\sigma G_i (\sigma,r,\gamma) = n_i( \sigma)  \qquad  \text{ for }\sigma \in \Gamma_i \,\,\, (i =1,2,3).
$$
\end{lem}
\begin{proof}
It is easy to see that 
$$
\sigma - O_i = - \frac{1}{\kappa_i} n_i(\sigma) \qquad \text{ for } \sigma \in \Gamma_i  \quad (i = 1,2,3)\,. 
$$
 Using this, we  calculate
\begin{align*}
r\sin ( \gamma + \tfrac{ \pi}{3})\nabla_\sigma G_2 ( \sigma,r, \gamma) &= \sin(\gamma - \tfrac{\pi}{3})\sigma - r\sin( \tfrac{\pi}{3})(1, 0 ) 
\\
&=  \sin( \gamma - \tfrac{\pi}{3}) \big( \sigma - O_2 \big)\\
&=- \tfrac{ \sin ( \gamma - \frac{ \pi}{3})}{\kappa_2}  n_2 (\sigma) & &\text{for } \sigma \in \Gamma_2 \,. 
\end{align*}
Similarly we get 
\begin{align*}
 r\sin ( \gamma + \tfrac{ \pi}{3})\nabla_\sigma G_1 (\sigma, r,\gamma)
&=-  \tfrac{\sin( \gamma + \tfrac{\pi}{3})}{\kappa_1} n_1(\sigma)  & &\text{for } \sigma \in \Gamma_1. \\
r\sin ( \gamma + \tfrac{ \pi}{3})\nabla_\sigma G_3 (x, r,\gamma)
&= - \tfrac{ \sin ( \gamma - \pi)}{\kappa_3}  n_3 (\sigma) & &\text{for } \sigma \in \Gamma_3\,.
\end{align*}
Since $ r \sin( \gamma + \frac{\pi}{3})= - \frac{\sin( \gamma + \frac{\pi}{3})}{ \kappa_1}  $, by the identity \eqref{Eq: D useful one} we complete the proof.
\end{proof}
Furthermore, the following result holds. 
\begin{pr}\label{Prop: derivative G r} Let
 $\Gamma = DB_{r, \gamma, 0}(0, 0)$. Then
\begin{equation*}
\left \{ 
\begin{aligned}
\partial_r G_1( \sigma, r,\gamma) &=-1 &&\text{for } \sigma  \in \Gamma_1 ,  \\
\partial_r G_2( \sigma, r,\gamma) &= -\tfrac{1}{\sin ( \gamma + \frac{\pi}{3})} \Big( \tfrac{\sin (\tfrac{\pi}{3}) }{r} \sigma_1 + \sin(\gamma - \pi) \Big)  &&\text{for } \sigma  \in \Gamma_2, \\
\partial_r G_3( \sigma, r,\gamma) &= +\tfrac{1}{\sin ( \gamma + \frac{\pi}{3})} \Big( \tfrac{\sin (\tfrac{\pi}{3}) }{r} \sigma_1 - \sin(\gamma - \tfrac{\pi}{3}) \Big) &&\text{for } \sigma \in \Gamma_3.  
\end{aligned}
\right.
\end{equation*}
\begin{proof}
According to Lemma \ref{Rep: SPDB}, we have
$$
G_i(\sigma, r, \gamma) = 0 \qquad \text{ for } \sigma \in \Gamma_i \,.
$$
Therefore, differentiating with respect to  $r$ in the definitions of functions $G_i$, we observe  
\begin{align*}
-r\sin ( \gamma + \tfrac{ \pi}{3})\partial_r G_1( \sigma, r,\gamma)  &=  \sin ( \gamma + \tfrac{ \pi}{3})r && \text{for } \sigma \in \Gamma_1 , \\
-r\sin ( \gamma + \tfrac{ \pi}{3})\partial_r G_2(\sigma, r, \gamma)  &= \sin( \tfrac{\pi}{3}) \big\langle \sigma, (1, 0) \big \rangle +\sin (\gamma- \pi) r && \text{for } \sigma \in \Gamma_2 \,,\\
r\sin ( \gamma + \tfrac{ \pi}{3})\partial_r G_3(\sigma, r, \gamma)   &=   \sin( \tfrac{\pi}{3}) \big\langle \sigma, (1, 0) \big \rangle -\sin (\gamma - \tfrac{\pi}{3}) r&& \text{for } \sigma \in \Gamma_3\,,
\end{align*}
which finishes the proof.
\end{proof}
\end{pr}
Similarly we get
\begin{pr} \label{Prop: derivative gamma}
Let
 $\Gamma = DB_{r, \gamma, 0}(0, 0)$. Then
\begin{equation*}
\left \{ 
\begin{aligned}
\partial_{\gamma} G_1( \sigma, r,\gamma) &=0 &&\text{for } \sigma  \in \Gamma_1 ,  \\
\partial_\gamma G_2( \sigma, r,\gamma) &=\tfrac{1}{2r\sin ( \gamma + \frac{\pi}{3})}\Big(\cos( \gamma - \tfrac{\pi}{3})|\sigma|^2 - r^2\cos (\gamma- \pi)   \Big) &&\text{for } \sigma  \in \Gamma_2, \\
\partial_\gamma G_3( \sigma, r,\gamma) &= \tfrac{1}{2r\sin ( \gamma + \frac{\pi}{3})}\Big(\cos( \gamma - \pi)|\sigma|^2 - r^2\cos (\gamma-  \tfrac{\pi}{3})  \Big) &&\text{for } \sigma \in \Gamma_3.  
\end{aligned}
\right.
\end{equation*}
\end{pr}
\subsubsection{Five-dimensional smooth manifold}
Throughout this section, without loss of generality, we may assume that the center of  $\Gamma^*_1$ is at the origin of $\mathbb R^2$ and that the angle $\theta^* = 0$, that is  
$$
\Gamma^* = {DB}_{r^*, \gamma^*,0}(0,0)\,.
$$ 

Clearly,  $\mathcal E \neq \emptyset$ as $\rho \equiv 0$ parameterizes $\Gamma^* = DB_{r^*, \gamma^*, 0}(0, 0)$. First we   demonstrate,  by applying the  implicit function theorem,
 that    every  standard planar double bubble $DB_{r, \gamma, \theta}( a_1, a_2)$ sufficiently close to $\Gamma^* = DB_{r^*, \gamma^*, 0}(0, 0)$ can be parameterized by some unique vector function $\rho= (\rho_1, \rho_2, \rho_3)$    depending smoothly on the parameters $a_1, a_2, r_{}, \gamma_{}$ and $ \theta $. We continue then to verify that  the set $\mathcal E$ of equilibria is actually a smooth manifold of dimension five.  
\begin{theo}
Any standard planar double bubble $DB_{r, \gamma, \theta}(a_{1},a_{2})$ sufficiently close to $\Gamma^*$, i.e., $(a_1, a_2, r, \gamma, \theta ) \in B_\epsilon (0, 0, r^*, \gamma^*,0 )$ for sufficiently small $\epsilon$, can be  parameterized by some unique smooth vector function $\rho=\rho(a_1, a_2, r, \gamma, \theta ) \in B_{X_1}(0,R)$.
\end{theo}
\begin{proof}
We use the implicit function theorem of Hildebrandt and Graves, see Zeidler \cite[Theorem 4.B]{Zeidlernonlinear1}, with  $ (x_0,y_0) =\big((0,0, r^*,\gamma^*, 0),0\big) $,
\begin{align*}
X &= \mR^2 \times B_{\delta_1} (r^*) \times B_{\delta_2}(\gamma^*)\times \mR \,,  \quad Z= Y,     
\\
Y&=  \Big \{ \rho \in   C^{4 + \a}(\Gamma_1^*) \times C ^ {4 + \a} ( \Gamma_2^*) \times C^{4 + \a} (\Gamma_3^*): \rho_1 + \rho_2 + \rho_3 = 0  \text{ on } \Sigma^* \Big \}, 
\end{align*}
\begin{align*} 
 F : X \times Y  &\rightarrow  Z\,,\\\big( (a_1, a_2, r, \gamma, \theta), \rho \big) &\mapsto \left(F_1, F_{2}, F_3 \right)
\end{align*}
with 
\begin{equation*}
   \begin{aligned}
      F_i( a_1, a_2, r, \gamma, \theta, \rho) &:= G_i \Big( Q_{\theta} T_{\vec  a}\Psi_ i(\cdot, \rho_i, \mu_i \circ \mathrm{pr}_i), r,\gamma \Big )  \qquad (i=1,2,3)\,.
\end{aligned}
\end{equation*}
 Here  $G_i$ are the functions stated in  Lemma \ref{Rep: SPDB} and

\begin{align*}
   \Psi_i(\cdot,\rho_{i},\mu_i \circ \mathrm{pr}_i)(\sigma) = \sigma + \rho_i ( \sigma ) n^*_i(\sigma )+ \mu_i (\mathrm{pr}_i(\sigma) ) \tau_i^*(\sigma)  & &\text{ for }\sigma \in \Gamma^*_i\,,
\end{align*}
where $ \mu  =   \mathcal{J} \rho$ on $ \Sigma^*$.  
Furthermore, 
\begin{align*} 
   Q_\theta     = 
               \begin{bmatrix}
                 \cos \theta   & \sin \theta \\
                  -\sin \theta & \cos \theta 
               \end{bmatrix} , 
   \quad  T_{\vec a} v = v - \vec a               
\end{align*}  
are the clockwise rotation matrix and  the translation operator respectively.

Indeed, the   image of the function $F$ lies in   $Z=Y$, that is    
\begin{equation} \label{Id: sum}
F_1 + F_2 + F_3 = 0 \quad\text{ on } \Sigma^*.
\end{equation}
To see this, note  that for  $ \sigma \in\Sigma^*$, 
$$\Psi_1(\cdot,\rho_{1},\mu_1 \circ \mathrm{pr}_1)(\sigma)  = \Psi_2(\cdot,\rho_{2},\mu_2 \circ \mathrm{pr}_2)(\sigma)  = \Psi_3(\cdot,\rho_{3},\mu_3 \circ \mathrm{pr}_3)(\sigma) \,,  $$ 
by Lemma \ref{Con:contact}. This together with the identity  \eqref{Id: sum to zero} proves \eqref{Id: sum}.

Moreover, since $\Psi_i|_{\rho = 0} = I$, according to Lemma \ref{Rep: SPDB} we have  
$$
F_{i}( x_0, y_0 )(\sigma)=F_{i} \big((0,0, r^*,\gamma^*, 0),0\big) (\sigma)= G_i (\sigma, r^*,\gamma^* ) = 0 \quad \text{ for } \sigma \in \Gamma_i^* .
$$
Thus $ F(x_0,y_0) = 0 $.
Now let us   compute the derivative $\partial_\rho F (x_{0} , y_0)$: 
\begin{align*}
\partial_\rho F_{i} (x_{0}, y_0) (v)( \sigma ) &=  \nabla_\sigma G _i(\sigma, r^{*},\gamma^*) \cdot  \big( v_i^{} \, n_i^*(\sigma ) + \big(\mathcal J \,v ( \mathrm{pr}_i ( \sigma)) \big)_i  \,\tau_i^*(\sigma) \big) = v_i \,,
\end{align*}
where we used Lemma \ref{Cor: gradient G}. Thus 
\begin{equation}\label{Eq: bijec}
\partial_{\rho}F (x_0, y_0)= I \,. 
\end{equation}
Furthermore, $F$ is a smooth map on a neighborhood of $ (x_0, y_0)$.

Therefore, according to the implicit function theorem,      there exist neighborhoods $U = B_{\epsilon}(x_{0})$  of $x_0$ and $V=B_{X_1}(0, R)$ of $y_0 = 0$ and a smooth function
\begin{align*}
\rho: U         & \longrightarrow  V \\
      (a_1, a_2, r, \gamma, \theta )&\mapsto \rho (a_1, a_2, r, \gamma, \theta )\,,  
\end{align*}
such that 
$\rho(x_0) = 0$ and for  every $(a_1, a_2, r, \gamma, \theta ) \in B_\epsilon (0, 0, r^*, \frac{\pi}{3},0 )$ we have
\begin{align}\label{F: zero level set}
F\big((a_1, a_2, r, \gamma, \theta ), \rho(a_1, a_2, r, \gamma, \theta)\big)=0 \,.
\end{align}
Moreover if $(x, y) \in U \times V$ and $F(x , y) = 0$ then $y= \rho(x)$. 

We now claim that   $\Gamma_{\rho_{}^{}} = \{ \Gamma_{\rho_1^{}}, \Gamma_{\rho_2^{}}, \Gamma_{\rho_3^{}} \}$ parameterized by the  function   $\rho = \rho(a_1, a_2, r, \gamma, \theta )$  is the standard planar double bubble $DB_{r, \gamma, \theta}(a_{1},a_{2})$. 
To see this, note 
\begin{align*}
F_i\big((a_1, a_2, r, \gamma, \theta ), \rho(a_1, a_2, r, \gamma, &\theta )\big ) = 0 \\
& \iff G_i \Big( Q_{\theta} T_{\vec  a}\Psi_i  (\cdot, \rho_i, \mu_i \circ \mathrm{pr}_i), r,\gamma \Big ) = 0 \\
           &\iff Q_{\theta} T_{\vec  a} \Gamma_{\rho_i} \subset S_{r_i}(O_i) \quad \text{by Lemma \ref{Rep: SPDB}}.
\end{align*}
Therefore, since Lemma \ref{Con:contact} guaranties that the curves $ \Gamma_{\rho_1^{}},  \Gamma_{\rho_2^{}},  \Gamma_{\rho_3^{}}$ meet at their boundaries, we end up with two choices: Either $\Gamma_{\rho_i} = \Gamma_i^{}$, where $\Gamma = \{ \Gamma_1, \Gamma_2, \Gamma_3  \}$ is a standard double bubble $DB_{r, \gamma, \theta}(a_1,a_2)$ or $\Gamma_{\rho_i}$ is the complementary part of $\Gamma_i$ in $S_{r_i}(O_i)$. But the latter can not happen since the norm of $\rho$ is small. Hence  
$$
\Gamma_{\rho(a_1, a_2, r, \gamma, \theta )}= DB_{r, \gamma, \theta}(a_1,a_2) \,,
$$
as required.
\end{proof}
\begin{theo}\label{Theo: DB equilibria 5}
The set of equilibria $ \mathcal{E}$ is in a neighborhood of zero a $C^{2}$-manifold  in $ X_1 $ of dimension $5$.
\end{theo}
\begin{proof}
Remind that we have shown 
\begin{align*}
 \mathcal E \cap U &= \Big\{\rho \in  B_{X_1}(0,R) : \rho \text{ parameterizes a standard planar double bubble}\Big\} \cap U\\
&= \Big\{ \, \rho(a_1,a_2, r, \gamma, \theta): (a_1,a_2, r, \gamma, \theta) \in U = B_\epsilon (0,0, r^{*}, \gamma^*, 0)  \,\Big\}\,,
\end{align*}
where the function 
\begin{align*}
\rho:\,  U        &\longrightarrow   X_1 = C^{4 + \a}(\Gamma_1^*) \times C ^ {4 + \a} ( \Gamma_2^*) \times C^{4 + \a} (\Gamma_3^*) \\
     \quad (a_1, a_2, r, \gamma, \theta )& \mapsto \rho\ (a_1, a_2, r, \gamma, \theta )  
\end{align*}
is smooth, in particular $C^2$ and  $\rho(U) = \mathcal E$, $\rho(x_0) =\rho (0,0, r^{*}, \gamma^*, 0)= 0$.

Therefore,  it is left to check  that the rank of $\rho'(x_0)$ equals five (see  the definition of a manifold on page  \pageref{set of equilibria E}). To do this, we differentiate \eqref{F: zero level set}  with respect to $\iota \in \{a_1, a_2, r, \gamma, \theta \}$ and evaluate  at $x_0 $   to get 
\begin{align*} 
\partial_\iota F(x_{0}, 0) + \partial_\rho F(x_{0}, 0) \partial_\iota \rho (x_0) = 0 \,.
\end{align*}
Therefore,
\eqref{Eq: bijec} gives
$$
\partial_\iota \rho(x_0) =-   \partial_\iota F(x_0, 0)\qquad (\iota \in \{a_1, a_2, r, \gamma, \theta \})\,. 
$$

We now calculate  
\begin{equation*}
\begin{aligned}
\partial_{a_1}F_i(x_0, 0) &= \nabla_\sigma G_{i} (\sigma^{},r^*, \gamma^*)\cdot(-1,0)=n_i^*(\sigma) \cdot (-1,0) =\cos (\kappa_i^* x)\,,
\end{aligned}
\end{equation*}
where we used the fact $ n_i^* ( \sigma) =- (\cos (\kappa_i^* x), \sin( \kappa_i^* x))$, $i=1,2,3 $. Thus
$$
\partial_{a_1} \rho(x_0) =
\big( \cos (\kappa_1^* x ),  \cos (\kappa_2^* x ), \cos (\kappa_ 3^* x ) \big).
$$
Similarly,  we get $ 
\partial_{a_2} \rho(x_0) =
\big(\sin (\kappa_1^* x ),  \sin (\kappa_2^* x ), \sin (\kappa_3^* x ) \big)
$.

Next we calculate
\begin{equation*}
\begin{aligned}
\partial_{\theta}F_{i}(x_0, 0) &= \nabla_\sigma G_{i} (\sigma,r^*, \gamma^*)\cdot
\big( 
\left[
\begin{smallmatrix}
0 & 1 \\
-1 & 0
\end{smallmatrix}
\right] \cdot \sigma \big) \\
&=   n^*_i( \sigma)  \cdot \sigma^\perp = n_i^*(\sigma) \cdot \big(- \frac{1}{\kappa_i^*}n^*_i(\sigma) + O^*_i \big)^\perp
\\
&= n_i^*( \sigma) \cdot {O_i^*}^\perp 
\end{aligned}
\end{equation*}
and so 
$$
\partial_{\theta}\rho(x_0) = \tfrac{\sin (\frac{ \pi}{3})}{\sin (\gamma^*)}r^* \begin{pmatrix}
 0 \\[4pt] 
 \frac{ \sin (\gamma^*)}{\sin(\gamma^* - \frac{ \pi}{3}) }\sin (\kappa_2^* x) \\[4pt]
 \sin (\kappa_3^* x ) \,
  \end{pmatrix}\,.
  $$

We now compute the derivative   $\partial_r F(x_0,0)= \partial_r G(\sigma, r^*, \gamma^*)$. According to Proposition
\ref{Prop: derivative G r}
$$
\partial_r G_2(\sigma, r^*, \gamma^*) = -\frac{1}{\sin ( \gamma^{*} + \frac{\pi}{3})} \Big( \frac{\sin (\tfrac{\pi}{3}) }{r^*} \sigma_1 + \sin(\gamma^* - \pi) \Big)\,.
$$ 
First we consider the case $\kappa_2^* \neq 0$. 
Employing the arc-length parameterization of $\Gamma_2^*$  derived in Proposition \ref{Pro: DB arc para} we obtain 
\begin{align*}
\partial_r G_2(\sigma, r^*, \gamma^*) &= -\frac{1}{
\sin ( \gamma^{*} + \frac{\pi}{3})} \Big( \frac{\sin (\tfrac{\pi}{3}) }{r^*} \sigma_1 + \sin(\gamma^* - \pi) \Big) \\
&= \frac{ \kappa_1^* \sin (\frac{\pi}{3})}{ \kappa_2^* \sin ( \gamma^* + \frac{\pi}{3})}\cos(\kappa_2^* x)-\frac{1}{\sin ( \gamma^{*} + \frac{\pi}{3})} \Big(\frac{\sin^2 (\frac{\pi}{3}) }{\sin(\gamma^* - \frac{\pi}{3})} + \sin(\gamma^* - \pi)  \Big) \\[3pt]
&=\frac{  \sin (\frac{\pi}{3})}{ \sin ( \gamma^* - \frac{\pi}{3})}\cos(\kappa_2^* x)-\frac{\sin ( \gamma^{*} + \frac{\pi}{3})}{\sin(\gamma^* - \frac{\pi}{3})}\,, \end{align*}
where  we applied the formula $ \sin^2 (x)-\sin^2 (y)= \sin(x+y)\sin(x-y)$.
 
A similar argument works for $\partial_r G_3(\sigma, r^*, \gamma^*)$.
Altogether we derive in case $\gamma^* \neq \frac{\pi}{3}$,
$$
\partial_{r} \rho(x_0) =   
\begin{pmatrix}
  1 \, \vspace{5 pt}\\ 
     -  \dfrac{  \sin (\tfrac{\pi}{3})}{ \sin ( \gamma^* - \frac{\pi}{3})}\cos(\kappa_2^* x)+\dfrac{\sin ( \gamma^{*} + \frac{\pi}{3})}{\sin(\gamma^* - \frac{\pi}{3})} \vspace{6pt}\\
    \dfrac{  \sin (\tfrac{\pi}{3})}{ \sin ( \gamma^* -\pi)}\cos(\kappa_3^* x)+\dfrac{\sin ( \gamma^{*} + \frac{\pi}{3})}{\sin(\gamma^* - \pi)} \,
  \end{pmatrix} \,.
$$
Next we consider the case $\kappa_2^* = 0$:
 We calculate \begin{align*}
 \partial_r G_2(\sigma, r^*, \tfrac{\pi}{3}) &= -\frac{1}{\sin ( \frac{2\pi}{3})} \Big( \frac{\sin (\tfrac{\pi}{3}) }{r^*}  \frac{r^*}{2^{}}  + \sin(- \frac{2\pi}{3}) \Big) = \frac{1}{2}\,.
\end{align*}
Therefore, we derive in case $\kappa_2^* = 0$, i.e., when $x_0 = (0,0, r^*, \frac{\pi}{3},0)$,
$$
\partial_{r} \rho(x_0) =   
\begin{pmatrix}
  1  \\ 
   -\frac{1}{2}  \\
   - \cos(\kappa_1^*x)-1
  \end{pmatrix}\,.
$$

Finally let us calculate  $   \partial_\gamma F(x_0, 0)$. We have
$
\partial_\gamma F(x_0,0) = \partial_\gamma G(\sigma, r^*, \gamma^*)$. We first consider the case $\kappa_2^* \neq 0$:
Employing the arc length parameterization of $\Gamma^*$ to the formulas derived in  Proposition
\ref{Prop: derivative gamma} we derive in case $\kappa_2^* \neq 0$ that 
$$
\partial_{\gamma} \rho(x_0) =   
\begin{pmatrix}
  0  \\ 
  a_2 \cos(\kappa_2^*x) +b_2 \\
   a_3 \cos(\kappa_3^*x) + b_3
  \end{pmatrix}
$$
for some constants $a_i, b_i$ (see the Appendix for the  explicit form of the constants). This immediately implies that   $\partial_\gamma \rho(x_0)$ is independent from the other elements of $\rho'(x_o)$. 

However,  we give the explicit formula in case  $\kappa_2^* = 0$.
   Using Proposition
\ref{Prop: derivative gamma}  we see
\begin{align*}
\partial_\gamma G_2( \sigma, r^*,\tfrac{\pi}{3}) &= \frac{1}{2r^*\sin ( \frac{2\pi}{3})} \Big(  \frac{1}{4} (r^*)^2 + x^2 + \frac{1}{2} (r^*)^2 \Big)\\
&=- \frac{\kappa_1^*}{\sin ( \frac{2\pi}{3})} \Big(\frac{1}{2}x^2 + \frac{3}{8}\frac{1}{(\kappa_1^*)^2} \Big)\,,
\end{align*}
\begin{align*}
\partial_\gamma G_3( \sigma, r^*,\tfrac{\pi}{3}) &=\frac{1}{2r^*\sin ( \frac{2\pi}{3})}\Big(\cos (-\tfrac{2\pi}{3})|\sigma|^2 - (r^*)^2  \Big) \\
&= \frac{-1}{2r^*\sin ( \frac{2\pi}{3})} (\frac{1}{2} | \sigma|^2+ (r^*)^2 \big) \\
&=  \frac{1}{ \kappa_1^*\sin ( \frac{2\pi}{3})}( 1 + \frac{1}{2} \cos(\kappa_1^* x))\,.
\end{align*}

In summary, we  have proved that  the rank of $\rho'(x_0)$ is equal to five and we have shown that the set of equilibria $ \mathcal{E}$ is a $C^{2}$-manifold  in $ X_1 $ of dimension five. Moreover 
$$
T_{0} \mathcal E = \mathrm{span}\big\{ v^{(1)}, v^{(2)}, v^{(3)}, v^{(4)}, v^{(5)} \big\},
$$
where
\begin{equation*}
v^{(1)} =
 \begin{pmatrix}
    \cos (\kappa_1^* x )\\[4pt]
    \cos (\kappa_2^* x )\\[4pt]
    \cos (\kappa_3^* x )
  \end{pmatrix}, \quad
v^{(2)} 
   =\begin{pmatrix}
    \sin (\kappa_1^* x )\\[4pt]
    \sin (\kappa_2^* x )\\[4pt]
    \sin (\kappa_3^* x )
  \end{pmatrix}, \quad
v^{(3)} 
  =\begin{pmatrix}
    0\\[4pt]
     \frac{ \sin(\gamma^*)}{ \sin(\gamma^*-  \frac{\pi}{3})}\sin (\kappa_2^* x )\\[4pt]
\sin (\kappa_3^* x )
  \end{pmatrix},
\end{equation*}
\begin{equation*}
v^{(4)}=
\begin{pmatrix}
  1 \\[2pt] 
       \frac{  -\sin (\frac{\pi}{3})}{ \sin ( \gamma^* - \frac{\pi}{3})}\cos(\kappa_2^* x)+\frac{\sin ( \gamma^{*} + \frac{\pi}{3})}{\sin(\gamma^* - \frac{\pi}{3})} \\[7.5pt]
    \frac{  \sin (\frac{\pi}{3})}{ \sin ( \gamma^* -\pi)}\cos(\kappa_3^* x)+\frac{\sin ( \gamma^{*} + \frac{\pi}{3})}{\sin(\gamma^* - \pi)} \,
  \end{pmatrix}\,, \quad
v^{(5)}=
\begin{pmatrix}
  0  \\[2pt]
  a_2 \cos(\kappa_2^ *x) +b_2 \\[6pt]
   a_3 \cos(\kappa_3^*x) + b_3
  \end{pmatrix}.
\end{equation*}
Although $v^{(i)}$ are continuous  in particular at $\kappa_2^* = 0$,  for  convenience we state them in case $\kappa_2^* =0$:
\begin{equation*}
v^{(1)} =
 \begin{pmatrix}
    \cos (\kappa_1^* x )\\
    1\\
    \cos (\kappa_1^* x )
  \end{pmatrix}, \quad
v^{(2)} 
   =\begin{pmatrix}
    \sin (\kappa_1^* x )\\
   0\\
    -\sin (\kappa_1^* x )
  \end{pmatrix}, \quad
v^{(3)} 
  =\begin{pmatrix}
    0\\
     \kappa_1^* x\\
-\sin (\kappa_1^* x )
  \end{pmatrix},
\end{equation*}
\begin{equation*}
v^{(4)}=
\begin{pmatrix}
  1  \\[2pt] 
      -\frac{1}{2}\\[3pt]
    - \cos(\kappa_1^*x)-1 \,
\end{pmatrix}\,, \quad
v^{(5)}=
\begin{pmatrix}
  0  \\[5pt]
   \tfrac{\kappa_1^*}{\sin ( \frac{\pi}{3})} (\frac{1}{2}x^2 + \frac{3}{8}\frac{1}{(\kappa_1^*)^2} )\, \\[9pt]
     \frac{-1}{ \kappa_1^*\sin ( \frac{\pi}{3})}( \frac{1}{2} \cos(\kappa_1^* x) + 1)
  \end{pmatrix}.  
\end{equation*}
\end{proof}
\subsubsection{Geometric interpretation of the null space} \label{Sec: gem interp} 
As an immediate corollary of Theorem \ref{Theo: DB equilibria 5}  we get
\begin{cor} \label{cor: null and manifold}
The null space $N(A_0)$ is five dimensional. Furthermore,
$$
T_0 \mathcal E = N(A_0)\,.
$$
\end{cor}
\begin{proof}
It always holds$$
T_0 \mathcal E  \subseteq N(A_0), 
$$
see \cite[equation (2.8)]{AbelsArabGarcke}.
Thus, according to Theorem \ref{Theo: DB equilibria 5} and Lemma \ref{Lem: DB null < 5},
$$
5=\dim (T_0 \mathcal E) \leq \dim (N(A_0)) \leq 5 \,.
$$
It follows that $\dim (N(A_0)) = 5$ and moreover $T_0 \mathcal E = N(A_0)$.      
\end{proof}
\noindent
\subsubsection*{Variations  preserving areas and curvatures}\textbf{}
We  easily see, using formula \eqref{Eq: D useful one}, that
\begin{equation*}
\left \{ 
\begin{aligned}
\int_{\Gamma_1^*}\! {v^{(1)}}_1 &= \int_{\Gamma_2^*} \! {v^{(1)}}_2 = \int_{\Gamma_3^*} \! {v^{(1)}}_3 = -2 \frac{\sin (\gamma^* + \frac{\pi}{3})}{\kappa_1^*} \, \,, \\
\int_{\Gamma_1^*}\! {v^{(2)}}_1 &= \int_{\Gamma_2^*} \! {v^{(2)}}_2 = \int_{\Gamma_3^*} \! {v^{(2)}}_3 = 0 \,,\\
\int_{\Gamma_1^*} \!{v^{(3)}}_1 &= \int_{\Gamma_2^*} \!{v^{(3)}}_2 = \int_{\Gamma_3^*} \! {v^{(3)}}_3 =0.   
\end{aligned}
\right.
\end{equation*}\label{Inc: jacobi}
In other words,
\begin{equation}
  \mathcal J (\Gamma^*) \subseteq \mathcal F( \Gamma^* ) .
\end{equation}
By Lemma \ref{Lem: variation}, each of the $v^{(i)}$ $(i = 1,2,3)$     corresponds to a  first variation of $\Gamma^*$ which preserves the areas, and  the curvatures to first order.
Indeed, we have demonstrated in the proof of Theorem \ref{Theo: DB equilibria 5}  that $v^{(1)}$, $v^{(2)}$, $v^{(3)}$ correspond to the first variations of the double bubble $\Gamma^*$ associated with translation along $x$-axis, translation along $y$-axis and   rotation around the center of $\Gamma_1^*$, respectively.
\subsubsection*{Variations not preserving  areas and curvatures}
It is shown  in the proof of Theorem \ref{Theo: DB equilibria 5}  that $v^{(4)}$ corresponds to the first variations of the double bubble $\Gamma^*$ associated with uniform scaling (with the scale factor $ \frac{r}{r^*}$). 
Let  $A_i(r)$ denote the area of the regions $R_i(r)$ corresponding to  the  double bubble $DB_{r, \gamma^*, \theta^*}( a_1^*, a_2^*)$. Then
(see equation (3.1) in \cite{HutchingsMorganRitorAntonio})  
\begin{align}\label{Eq: derivative area DB}
\partial_r A_1  = \int_{\Gamma_1^*}\! {v^{(4)}}_1 
- \int_{\Gamma_2^*}\! {v^{(4)}}_2 > 0 \,, \quad \partial_r A_2  = \int_{\Gamma_2^*}\! {v^{(4)}}_2 - \int_{\Gamma_3^*}\! {v^{(4)}}_3 >0 \, 
\end{align}
according to Lemma \ref{Lem: v4 negative} (ii).

Again remember from the proof of Theorem \ref{Theo: DB equilibria 5}  that  $v^{(5)} $ corresponds to the first variation of $\Gamma^*$ with  respect to the angle $\gamma$,  that is  w.r.t.  the curvature ratio.
Similarly we  denote by   $A_i(\gamma)$  the area of the regions $R_i(\gamma)$ corresponding to  the  double bubble $DB_{r^{*}, \gamma, \theta^*}( a_1^*, a_2^*)$. Then
\begin{align}\label{Eq: derivative area DB gamma}
\partial_\gamma A_1  = \int_{\Gamma_1^*}\! {v^{(5)}}_1 
- \int_{\Gamma_2^*}\! {v^{(5)}}_2 > 0\,, \qquad \partial_\gamma A_2  = \int_{\Gamma_2^*}\! {v^{(5)}}_2 - \int_{\Gamma_3^*}\! {v^{(5)}}_3 < 0\, 
\end{align}
according to Lemma \ref{Lem: v5 negative} (ii).

We now define  the matrix 
\begin{align*}
D :=
\begin{pmatrix}
\partial_r A_1 & \partial_\gamma A_1 \\[13pt]
\partial_r A_2 & \partial_\gamma A_2 
\end{pmatrix}
=
\begin{pmatrix}
\displaystyle  \int_{\Gamma_1^*}\! {v^{(4)}}_1 
- \int_{\Gamma_2^*}\! {v^{(4)}}_2 &  \displaystyle  \int_{\Gamma_1^*}\! {v^{(5)}}_1 - \int_{\Gamma_2^*}\! {v^{(5)}}_2 \\[13pt]
\displaystyle  \int_{\Gamma_2^*}\! {v^{(4)}}_2 - \int_{\Gamma_3^*}\! {v^{(4)}}_3  & \ \int_{\Gamma_2^*}\! {v^{(5)}}_2 - \int_{\Gamma_3^*}\! {v^{(5)}}_3
\end{pmatrix}
.
\end{align*}
\begin{lem} \label{Lem: matrix D invert}
The matrix $D$ is invertible for each $0< \gamma^* < \frac{2 \pi}{3}$.
\end{lem}
\begin{proof}
Let us  calculate its  determinant.  Inequalities  \eqref{Eq: derivative area DB} and \eqref{Eq: derivative area DB gamma}  imply
\begin{align*}
\det D &=\partial_r A_1 \partial_\gamma A_2 - \partial_\gamma A_1 \partial_r A_2 < 0
\end{align*}
Now as the determinant of the matrix $D$ is  strictly negative, we conclude that   the matrix $D$ is for each $0< \gamma^* < \frac{2 \pi}{3}$ invertible, \end{proof}
As a further  result of Lemma \ref{Lem: v4 negative} and \ref{Lem: v5 negative} (ii),    we get $v^{(4)}, v^{(5)} \notin \mathcal F( \Gamma^*)$. Therefore, we conclude from Lemma \ref{Lem: variation}  that  the corresponding variations do not preserve areas to first order. Indeed we  will show below in Lemma \ref{Lem: bilinear negative} that 
$$
I(u,u)<0 \qquad  \text{for } u=v^{(4)}, v^{(5)}\,.
$$
In addition  they do not preserve the curvatures to first order too as the constant vectors $c(v^{(4)})$ and $c(v^{(5)})$ are nonzero.
\subsection{Semi-simplicity} \label{Sec: semi-simplicity2}
We  need to show    two small propositions.
The first one is stated and proved in the proof of Lemma 3.8 in \cite{HutchingsMorganRitorAntonio}. 
\begin{pr} \label{Prop: bilinear zero}
If $u \in \mathcal F (\Gamma^*)$ satisfies $I(u,u)= 0$,   then 
$$
I(u,v) = 0 \qquad \forall \, v \in \mathcal F(\Gamma^*)\,.
$$ 
\end{pr}
\begin{proof}
According to Corollary \ref{Cor: D stable},   $  I(v + tu, v + tu ) \geq 0 $ for all $v \in \mathcal F(\Gamma^*)$ and $t \in \mR$. 
Therefore  
\begin{align*}
I(v+tu, v+ tu) &= I (v,v) + 2t I(u,v) + t^2 I(u,u)\\
               &= I (v,v) + 2t I(u,v)\,.  
\end{align*}
This forces $I(u,v) = 0$ as $t$ can take arbitrary negative values.  
\end{proof}
\begin{pr} \label{Prop: range and mean}
Let $z \in R(A_0)$. Then there exists $u \in  \mathcal F( \Gamma^*) \cap D(A_0)$ such that $A u = z$.
\end{pr}
\begin{proof}
Clearly, there exists $ \widetilde u \in D(A_0) $ such that $ A \widetilde u = z $.
 The actual task is to find  two  constants $\a( \widetilde u) $, $\beta(\widetilde u)$  such that $$
u:=  \widetilde u + \a( \widetilde u) v^{(4)} + \beta( \widetilde u) v^{(5)} 
$$
 satisfies  
 $$
 \int_{\Gamma^*_1} \! u_1 = \int_{\Gamma_2^*} \! u_2 = \int_{\Gamma_3^*} \! u_3 \,. 
 $$
(This  will finish the  proof since  $v^{(4)}, v^{(5)} \in N(A_0)$ implies $Au = A \widetilde u = z$.) To do so,  let us recast this integral constraint into the matrix form   $$
 D 
\begin{pmatrix}
\displaystyle \a (\widetilde u) \\[13pt]
\displaystyle \beta (\widetilde u)
\end{pmatrix}
= 
\begin{pmatrix}
\displaystyle\int_{\Gamma_1^*} \! {\widetilde u}_2 - \int_{\Gamma_2^*} \! {\widetilde u}_1  \\[13pt]
\displaystyle\int_{\Gamma_1^*} \! \widetilde u_3 - \int_{\Gamma_3^*} \! \widetilde u_2 
\end{pmatrix}\,,
 $$ 
where the matrix $D$ is given above.
The invertibility of this matrix proved in Lemma \ref{Lem: matrix D invert} finishes the
proof.\end{proof}
We are now ready to prove:
\begin{lem}
The eigenvalue $0$ of $A_0$ is  semi-simple.
\end{lem}
\begin{proof}
Since the  operator
$A_0$ has a compact resolvent, the semi-simplicity condition
is equivalent to the condition that $N(A_0) = N(A_0^2)$ (use the spectral theory
of compact operators, e.g. see [4, Section 9.9]). In other words, it suffices  to check that 
 $$
 R(A_0) \cap N(A_0) = \{ 0 \}\,.
 $$
To prove this, let $z \in R(A_0) \cap N(A_0)\,\big( \subset D(A_0) \big)$. 
 According to  Proposition \ref{Prop: range and mean} there exists $u \in \mathcal   D(A_0) \cap \mathcal F( \Gamma^* ) $ such that 
$
A u = z 
$.
From this, exactly as  done in  Section \ref{Sec. DB spectrum}, we derive the identity   
\begin{align} \label{Eq: range and bilinear}
 \sum_{i =1}^3  \int_{\Gamma_i^*} \big|\nabla_{\Gamma_i^*} \big(  \Delta_{\Gamma_i^*} u_i + (\kappa ^*_i)^2 u_i \big) \big|^2  \mathrm d s
&=  I(z,u) \,,
\end{align}
where we  used only the facts that $u, z \in D(A_0)$.

Moreover, similarly as before, an  integration    and application of the divergence theorem using the fact that $u \in D(A_0)$ gives
\begin{equation*}
\int_{\Gamma^*_1} \!z_1 = \int_{\Gamma_2^*}\! z_2 = \int_{\Gamma_3^*} \! z_3 \,,
\end{equation*}
and so $z \in \mathcal F ( \Gamma^*)$.

Now since $z \in N(A_0) \cap  \mathcal F( \Gamma^*)$,  Corollary \ref{Cor: null bilinear}  tells us   $I(z,z) = 0$.  Therefore, according to Proposition \ref{Prop: bilinear zero}, 
$$
I(z,u) =0 
$$
as $u \in \mathcal F( \Gamma^*)$. In view of the identity \eqref{Eq: range and bilinear}, we obtain  $u \in N(A_0)$. Consequently  $z=Au=0$, which finishes the proof.    
\end{proof}

\begin{rem}
The main ingredient in the proof is the positivity of the bilinear form, i.e., the variational stability of the  stationary solution.\end{rem}
\section{Standard planar double bubbles are stable} \label{Sec: main theo}
Summing up, we have shown that  all the hypotheses  of Theorem \ref{theo01} are satisfied. Thereby   applying  Theorem \ref{theo01} we conclude:
\begin{theo}\label{main theorem example}
Let $\Gamma^*$  be a standard planar double bubble. Then $\rho^*\equiv 0$ is a stable equilibrium of \eqref{Eq: D nonlinear semifinal} in $X_1 = C^{4 + \a }(\o \,,\, \mR^3)$.
 Moreover, if $\rho^0$ is sufficiently close to $\rho^* \equiv 0$ in $X_1$ and satisfies the corresponding compatibility conditions \eqref{compatibility}, then the problem \eqref{Eq: D nonlinear semifinal} has a unique solution 
$$
 \rho \in C^{1+\frac{\alpha}{2m}, 2m+\alpha} \big( [0, \infty)\times\o \,,\, \mR^3 \big)
$$
and approaches some  $\rho^\infty\in \mathcal{E}$, parameterizing some standard planar double bubble, exponentially fast in $X_1$ as $t \to \infty$.
\end{theo}
In this sense, the standard planar double bubble $\Gamma^*$ is  stable under the surface diffusion flow. In addition, every planar double bubble that starts sufficiently close to $\Gamma^*$ and satisfies the angle, curvature,  balance of flux condition and   the condition on the Laplace of the curvatures, see \eqref{Con: sum laplace curvature}, at $t = 0$  exists globally and converges to some standard planar double bubble, enclosing the same areas as its initial data, at an exponential rate as $t\rightarrow \infty$.
 We illustrate this result in Figure \ref{Fig: stability}. 
\subsection{General area preserving gradient flows}\label{Sec: general flow} 
It is to be expected that for any sufficiently smooth  area preserving gradient flow  
\begin{equation*}
V = - \mathrm{grad}_{\mathcal H(\Gamma)} \mathrm{Length}
\end{equation*}
one  obtains  the following identity
\begin{align} \label{Eq: C general bilinear}
\norm{z}_{\mathcal H({\Gamma}^*)}^2= { I(z,u)}_{}^{}\,,  
\end{align}
where
$ z := \delta \big(\mathrm{grad}_{\mathcal H(\Gamma)} 
\mathrm{Length \big)}(u) $.
Here $\mathcal H(\Gamma)$ denotes a (pre-)Hilbert manifold with some area constraints. 

In particular,   if $u$ is a eigenvector of the operator $\delta \big(\mathrm{grad}_{\mathcal H(\Gamma)} 
\mathrm{Length \big)}$ with respect to the eigenvalue $\lambda$, then we get 
\begin{align} \label{Eq: C2 general bilinear}
  { \Big\|\delta \big(\mathrm{grad}_{\mathcal H(\Gamma)} 
\mathrm{Length \big)}(u)\Big \|}_{\mathcal H(\Gamma^*)}^2 = \lambda I(u,u)\,. \qquad 
\end{align}
 
Comparing the   identifies \eqref{Eq: C2 general bilinear} and \eqref{Eq: C general bilinear} with the identities \eqref{Eq: D bilinear} and \eqref{Eq: range and bilinear} respectively, we expect that our approach can be used for other area preserving gradient flows. Therefore we   conjecture that

\begin{conj} \label{Conj}
  Standard planar double bubbles are stable under  sufficiently smooth    area preserving gradient flows.
\end{conj}
It would be desirable to analyze the problem systematically.

\begin{appendices}
\section{More about the bilinear form $I(\,,\,)$}\label{Sec: more about bilinear}
\begin{lem} 
Within the class of functions $u$ satisfying the linearized angle condition, we have
\begin{align*}
\big \{ \, u : I(u,u)= 0 \,\big\} \cap \mathcal F(\Gamma^*)=  N(A_0) \cap \mathcal F(\Gamma^*).
\end{align*}
\end{lem}
\begin{proof}
We first assume $u \in  \mathcal F(\Gamma^*)$ such that $I(u,u) = 0$.  Then by   Lemma 3.8 in \cite{HutchingsMorganRitorAntonio} and the fact that $u$ satisfies the linearized angle condition, we conclude that $u \in N(A_0)$.
The converse statement is  Corollary  \ref{Cor: null bilinear}.
\end{proof}
Note that  we have already shown in  Section \ref{Sec: gem interp} that 
$$
N(A_0) \cap \mathcal F(\Gamma^*)= \mathcal J (\Gamma^*)= \mathrm{span} \{  v^{(1)}, v^{(2)}, v^{(3)} \}.
$$
On the other hand  we  obtain: 
\begin{lem} \label{Lem: bilinear negative}
For the bilinear form $I$ it holds
\begin{equation*}
\begin{aligned}
 I ( u, u ) < 0   & &\text{ for } u = v^{(4)}, v^{(5)}\qquad (0< \gamma^* < \tfrac{2 \pi }{3})\,. 
\end{aligned} 
\end{equation*} 
\end{lem}
\begin{proof}
According to Lemma \ref{Lem: null bilinear},
$$
I(u,u) =-\sum_{i=1}^3  c_i(u) \int_{\Gamma_i^*} u_{i} \qquad \text{ for } v^{(4)}, v^{(5)}\big(\in N(A_0)\big) \,.
$$ 
Now assertion (iii) in Lemma \ref{Lem: v4 negative} and Lemma \ref{Lem: v5 negative}   proves the lemma.  
\end{proof}
\section{The proof of Lemma \ref{Rep: SPDB}} \label{App. Rep: SPDB}
 Consider the standard planar double bubble  $\Gamma = DB_{r, \gamma, 0}(0, 0) $. That is  the left circular arc $\Gamma_1$ has radius $r_1 = r$  centered at $O_1= (0,0)$ and all the other centers also lie  on the $x$-axis, for some  $ r > 0 $, $0 < \gamma < \tfrac{2 \pi}{3}$, see Figure \ref{Fig: Morgan}.

It follows directly from the law of sines that in case $\gamma \neq \frac{ \pi}{3}$ 
\begin{align*}
O_{2} &= \Big(  \frac{\sin (\tfrac{2 \pi}{3}) }{\sin (\gamma - \tfrac{ \pi}{3})}  \, r\, ,\,0 \Big) ,   && r_2 = \Big|\frac{\sin (\tfrac{2 \pi}{3} - \gamma) }{\sin (\gamma - \tfrac{ \pi}{3})}\Big| \,r\,,  \\
 O_{3} &= \Big( \frac{\sin (\tfrac{ \pi}{3})}{\sin (\gamma)}\, r \, , \, 0 \Big) \,, && r_3= \frac{\sin (\tfrac{2 \pi}{3} - \gamma) }{\sin (\gamma )}r\,.
\end{align*}

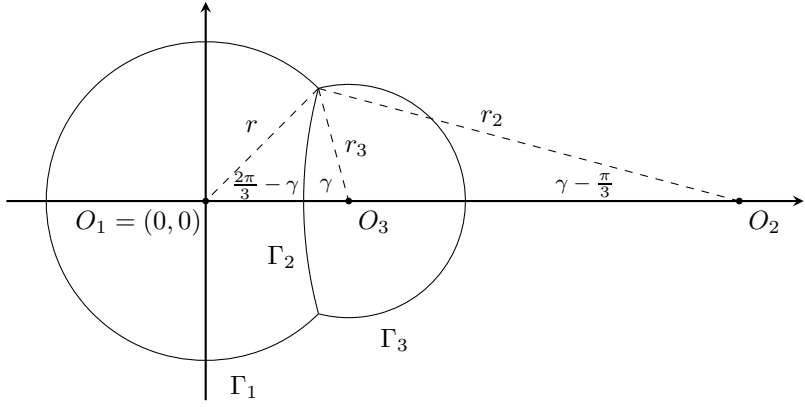
\begin{figure}[htbp]
           \centering
        \begin{tikzpicture}[scale=0.53,>=stealth]
                \draw (2.828,2.828) arc (45:315:4) node[above=-0.97cm,right=-1.3cm]{$ \Gamma_1$};
                \draw (2.828,2.828) arc (105:-105:2.928)node[above=-0.35cm,right=0.7cm]{$ \Gamma_3$};
                \draw (2.828,2.828) arc (165:195:10.92) node[above=0.75cm,right=-0.8cm]{$ \Gamma_2$};
                \filldraw [black] (0,0) circle (2pt) node[below= 8 , left= -2] {$O_1=(0,0)$}node[below=-6 , right= 6 ] {\footnotesize{$\tfrac{2 \pi}{3} - \gamma $}} node[above= 1cm , right = 11.5] {$r$} ;
               \filldraw [black] (3.5863,0) circle (2pt) node[below=8 , right= 0 ] {$O_3$}node[below=-6 , right=-14.5 ] {\footnotesize{$\gamma$}};
               \filldraw [black] (13.384,0) circle (2pt) node[below=8, right= 0 ] {$O_2$}node[below=-6 , right= -73 ] {\footnotesize{$\gamma - \tfrac{\pi}{3}$}}  node[above=1.1cm , left = 3cm ] {$r_2$};
             
                 \draw [dashed] (2.828,2.828) --  (3.5863,0)   node[above = 20 , right = -5]{$r_3$};
                 \draw [dashed] (0,0) --  (13.384,0)  node[above]{$$};
                 \draw [dashed] (2.828,2.828) --(0,0) node[right]{$$};
          
                    \draw[dashed] (2.828,2.828) --  (13.384,0)  node[left]{};
                     \draw [->,thick](0,-5)--(0,5);
                      \draw [->,thick](-5,0)--(15,0);
                        \end{tikzpicture}
                        \caption{Standard planar double bubble $\Gamma = DB_{r, \gamma, 0}(0, 0)$}
                        \label{Fig: Morgan}
\end{figure}
Therefore, for $\sigma = (\sigma_1, \sigma_2) \in \Gamma_2$, $\gamma \neq \frac{\pi}{3}$,
 we have
\begin{align*}
0 &= \Big |\sigma\ -\Big( \frac{\sin (\tfrac{2\pi}{3})}{\sin (\gamma - \tfrac{\pi}{3})} r, 0 \Big)  \Big|^2 -\Big( \frac{\sin ( \tfrac{2 \pi}{3} - \gamma) }{\sin (\gamma - \tfrac{ \pi}{3})}r \Big)^2 \\[5pt]
& = |\sigma|^2 -2 \sigma \cdot \Big(\frac{\sin (\tfrac{2\pi}{3})}{\sin (\gamma - \tfrac{\pi}{3})} r, 0 \Big) + \frac{\sin^2 (\tfrac{2\pi}{3}) - \sin^2 ( \tfrac{2 \pi}{3} - \gamma)}{\sin^2 (\gamma - \tfrac{\pi}{3})} r^2 \\
&=  \frac{2}{\sin( \gamma - \tfrac{\pi}{3})}  r\sin ( \gamma + \tfrac{ \pi}{3})G_2(\sigma ,r,\gamma)\,,
\end{align*}
where we  applied the formula $\sin^2 x - \sin^2 y = \sin(x+y)\sin(x-y)$.

Similarly,  for  $\sigma \in \Gamma_3$ we obtain
\begin{equation*}
0 = \Big |\sigma -\Big( \frac{\sin (\tfrac{\pi}{3})}{\sin (\gamma )} r, 0 \Big)  \Big|^2 -\Big( \frac{\sin ( \tfrac{2 \pi}{3} - \gamma) }{\sin (\gamma )}r \Big)^2 = \frac{2}{\sin( \gamma -  \pi)}  r\sin ( \gamma + \tfrac{ \pi}{3})G_3( \sigma ,r,\gamma)\,
\end{equation*}
and  obviously for $\sigma \in \Gamma_1$ we have 
$$
0 = |\sigma|^2 -r^{2} =\frac{2}{  \sin(\gamma + \tfrac{\pi}{3})}  r\sin ( \gamma + \tfrac{ \pi}{3})G_1 (\sigma, r,\gamma)\,.
$$
Furthermore, we see   for $\sigma \in \Gamma_2$, $\gamma = \frac{\pi}{3}$ that $ 0 = \frac{r}{2} - \sigma_1 =G_2(\sigma,r, \tfrac{\pi}{3} )$.

Finally, the identity \eqref{cos and sin} easily verifies   \eqref{Id: sum to zero}. This finishes the proof as the coefficients appearing above are all nonzero and well-defined.
\section{Arc-length parameterization of $\Gamma^*$}
\begin{pr}\label{Pro: DB arc para}
 An arc-length parameterization of $\Gamma_i^*$, $i = 1,2,3$ is given as follows: For $\gamma^* \neq \frac{\pi}{3}$, 
\begin{equation*}
(\sigma_1, \sigma_2)= \sigma =
\left \{ 
\begin{aligned}
&\frac{1}{\kappa_1^*} \big(\cos (\kappa_1^* x), \sin( \kappa_1^* x )\big) & & \text{for } \sigma \in \Gamma^*_1 \,,  \\[5pt]
& 
\Big(\frac{\sin (\tfrac{\pi}{3}) }{\sin(\gamma^* - \frac{\pi}{3})}  r^* + \frac{1}{\kappa_2^*} \cos (\kappa_2^* x) \,\,, \,\,  \frac{1}{\kappa_2^*}\sin(\kappa_2^* x) \Big) & & \text{for } \sigma \in \Gamma^*_2 \,,
\\[5pt]
&\Big (\frac{-\sin (\tfrac{\pi}{3}) }{\sin(\gamma^* -\pi)}  r^* + \frac{1}{\kappa_3^*} \cos (\kappa_3^* x) \,\,, \,\,  \frac{1}{\kappa_3^*}\sin(\kappa_3^* x) \Big)
& & \text{for } \sigma \in \Gamma^*_3 \,,
\end{aligned}
\right.
\end{equation*}
Moreover this arc-length parameterization is continuous  at $\gamma^* = \frac{\pi}{3}$ and in particular $\sigma = (\frac{r^*}{2} \,,  \,x)$ for $\sigma \in \Gamma^*_2$, $\gamma^* = \frac{\pi}{3}$.
\end{pr} 

\begin{proof}
We give the proof for $\Gamma^*_2$. We observe
\begin{align*}
(\sigma_1, \sigma_2)= \sigma &=  O^*_2 - \frac{1}{\kappa_2^*} n_2^* = O^*_2 + \frac{1}{\kappa_2^*} \big(\cos (\kappa_2^* x), \sin( \kappa_2^* x )\big) \quad  \\
&=
\bigg(\frac{\sin (\tfrac{\pi}{3}) }{\sin(\gamma^* - \frac{\pi}{3})}  r^* + \frac{1}{\kappa_2^*} \cos (\kappa_2^* x) \,\,, \,\,  \frac{1}{\kappa_2^*}\sin(\kappa_2^* x) \bigg) \qquad \text{for } \sigma \in \Gamma^*_2 \,.
\end{align*}
The proof of the continuity can be done using the  identity \eqref{Eq: D useful one}  and the  L'Hôpital's rule.
\end{proof}
\section{The signs of the integrals }
\begin{lem} \label{Lem: v4 negative}
Let $ 0< \gamma^* < \frac{2 \pi}{3}$.  Then 
\begin{align*}
\mathrm{(i)} & \; \int_{\Gamma_1^*} {v^{(4)}}_1 > 0 \,, \quad \int_{\Gamma_2^*} {v^{(4)}}_2 < 0 \,, \quad \int_{\Gamma_3^*} {v^{(4)}}_3 < 0 \,, 
\\[2pt]
\mathrm{(ii)}& \; \int_{\Gamma_1^*} {v^{(4)}}_1 - \int_{\Gamma_2^*} {v^{(4)}}_2 > 0 \,, \quad \int_{\Gamma_2^*} {v^{(4)}}_2 - \int_{\Gamma_3^*} {v^{(4)}}_3 > 0 \,, \\
\mathrm{(iii)}& \; \sum_{i=1}^3  c_i(v^4) \int_{\Gamma_i^*} {v^{(4)}}_{i} > 0 \,.  
\end{align*}
\end{lem}
\begin{proof}

In order to easily see the general strategy of the proof, let us first verify the assertions for   $\gamma^*= \frac{\pi}{3}$:   
\begin{align*}
\int_{-l^*_1}^{l^*_1} {v^{(4)}}_1 &= \int_{-l^*_1}^{l^*_1} 1 = 2 l^*_1 >0 \,, & &c_1(v^{(4)}) = (\kappa_1^*)^2 > 0 \,,\\
\int_{-l^*_2}^{l^*_2} {v^{(4)}}_2 &=- \int_{-l^*_2}^{l^*_2} \frac{1}{2} = -l^*_2 < 0 \,, & & c_2(v^{(4)}) = 0 \,, \\
\int_{-l^*_3}^{l^*_3} {v^{(4)}}_3 &=- \int_{- l^*_3}^{l^*_3} 1 +\cos(\kappa_1^*x)
<0 \,, & & c_3(v^{(4)}) = -(\kappa_3^*)^2 <0 \,.
\end{align*}
Therefore, 
$$
\sum_{i=1}^3  c_i(v^{(4)}) \int_{\Gamma_i^*} \!{v^{(4)}}_i > 0 \qquad (\gamma^* = \tfrac{\pi}{3})\,.
$$ 
Next we calculate  
\begin{align*}
\int_{\Gamma_2^*} {v^{(4)}}_2 - \int_{\Gamma_3^*} {v^{(4)}}_3 &= -l_2^* + 2 l_1^* + \int_{-l_1^*}^{l_1^*} \cos(\kappa_1^* x)= -l_2^* + 2 l_1^*  + 2 \frac{\sin(\kappa_1^* l^*_1)}{\kappa_1^*}  \\
&=\frac{\sin(\frac{ \pi}{3})}{\kappa_1^*} + 2 l_1^* -2 \frac{\sin(\frac{ \pi}{3})}{\kappa_1^*} = 2l_{1}^* +r^{*} \sin(\tfrac{ \pi}{3})  > 0 \,, 
\end{align*}
where we have used the facts that 
$$
l_1^* = l_3^* , \quad l_2^* = - \frac{\sin(\frac{ \pi}{3})}{\kappa_1^*}, \quad \kappa_1^* l_1^* = - \tfrac{2 \pi}{3} \qquad (\gamma^* = \tfrac{\pi}{3})\,.
$$
Obviously $\int_{\Gamma_1^*} {v^{(4)}}_1 - \int_{\Gamma_2^*} {v^{(4)}}_2 > 0$ which completes the proof of assertions  (i)-(iii)  in case $\gamma^* = \frac{\pi}{3}$.

Assume now $\gamma^* \neq \tfrac{\pi}{3}$.  Then we calculate  
$
\int_{\Gamma_1^*} {v^{(4)}}_1 =\int _{\Gamma_1^*} 1=  2 l^*_1 > 0 $ and
\begin{align*}
\frac{1}{2}\int_{-l^*_2}^{l^*_2} {v^{(4)}}_2 &= \frac{1}{2} \int_{-l^*_2}^{l^*_2}   \Big(\frac{  -\sin (\frac{\pi}{3})}{ \sin ( \gamma^* - \frac{\pi}{3})}\cos(\kappa_2^* x) + \frac{\sin ( \gamma^{*} + \frac{\pi}{3})}{\sin(\gamma^* - \frac{\pi}{3})} \Big) \\
&=  \frac{  -\sin (\frac{\pi}{3})}{\sin ( \gamma^* - \frac{\pi}{3})} \frac{\sin(\kappa_2^* l^*_2)}{\kappa_2^*} +  \frac{ \sin ( \gamma^{*} + \frac{\pi}{3}) }{\sin ( \gamma^* - \frac{\pi}{3})}l^*_{2  } =  \frac{\sin (\frac{\pi}{3})}{ \kappa_2^*}+ \frac{ \sin ( \gamma^{*} + \frac{\pi}{3}) }{\sin ( \gamma^* - \frac{\pi}{3})} l^*_2\\
&=  l^*_2 \Big(  \frac{-\sin (\frac{\pi}{3})}{  (\gamma^* - \frac{\pi}{3})}+ \frac{ \sin ( \gamma^{*} + \frac{\pi}{3}) }{\sin ( \gamma^* - \frac{\pi}{3})} \Big) = : l^*_2 f(\gamma^*)\,, 
\end{align*}
where we  used the fact that $\kappa_2^* l^*_2 = - (\gamma^* - \frac{\pi}{3})$. Similarly 
$$
\int_{-l^*_3}^{l^*_3} {v^{(4)}}_3 =2 l^*_3 \Big(  \frac{\sin (\frac{\pi}{3})}{  (\gamma^* -\pi)}+ \frac{ \sin (\gamma^* + \frac{\pi}{3}) }{\sin (\gamma^*- \pi)} \Big)= :2l_3^*  \,g( \gamma^*)\,.
$$
 Obviously the function $g$ is  negative  on $(0, \frac{2 \pi}{3})$. Taking into account the fact that $\sin(x)< x$ for  $ 0<x<\pi$, it is easy  to check that  $f(0) < 0$, $f'<0$   and so   $f<0$   on $(0, \frac{2 \pi}{3})$ too. Thus    
$$
\int_{-l^*_2}^{l^*_2} {v^{(4)}}_2 = 2 l^*_2 f(\gamma^*)< 0 ,  \quad \int_{-l^*_3}^{l^*_3} {v^{(4)}}_3 = 2 l^*_3 g(\gamma^*)< 0 \,.
$$
Assertion (i) follows. 

Similar argument shows that $g'>0$  and so $f'-g'<0$ on $(0, \tfrac{2 \pi }{3})$. This together with  $ (f-g)(\frac{2 \pi}{3}) = 0 $ implies $f-g>0$.
Observe further that   
$$
\frac{ \sin(\gamma^* - \pi)}{(\gamma^* - \pi)} \leq \frac{\sin(- \frac{\pi}{3})}{( - \frac{\pi}{3})} = \frac{\sin( \frac{\pi}{3})}{( \frac{\pi}{3})} < \frac{\sin(\gamma^* - \frac{\pi}{3})}{(\gamma^* - \frac{\pi}{3})}  \qquad \text{on } (0, \tfrac{2 \pi }{3})  
$$  
as the function $\frac{\sin(x)}{x}$ is  strictly increasing and decreasing  on intervals $(-\pi, 0)$ and $(0,\pi)$ respectively.  Thus we conclude $$
\frac{l^*_3}{l^*_2} =\frac{(\gamma^* - \pi)}{(\gamma^* - \frac{\pi}{3})}\frac{\kappa_2^*}{\kappa_3^*} = \frac{(\gamma^* - \pi)}{\sin(\gamma^* - \pi)}\frac{\sin(\gamma^* - \frac{\pi}{3})}{(\gamma^* - \frac{\pi}{3})} > 1 \, \qquad (0< \gamma^* < \tfrac{2 \pi}{3})\,. 
$$We are now ready to  estimate 
\begin{align*}
\int_{-l^*_2}^{l^*_2} {v^{(4)}}_2-\int_{-l^*_3}^{l^*_3} {v^{(4)}}_3 &= 2 l^*_2 f(\gamma^*) -2 l^*_3 g( \gamma^*) \\
&>2l^{*}_2 \big( f(\gamma^*) \,-g(\gamma^*) \big) > 0 \,. 
\end{align*}
Moreover, 
  $\int_{\Gamma_1^*} {v^{(4)}}_1 - \int_{\Gamma_2^*} {v^{(4)}}_2 > 0$ by assertion (i).  This proves assertion (ii).  
 
To prove assertion (iii) we observe:
\begin{align*}
c_2(v^{(4)}) = (\kappa_2^*)^2  \frac{ \sin ( \gamma^{*} + \frac{\pi}{3}) }{\sin ( \gamma^* - \frac{\pi}{3})}  = \kappa^*_1 \kappa_2^* \,, \quad  
c_3(v^{(4)}) = (\kappa_3^*)^2 \frac{ \sin ( \gamma^{*} + \frac{\pi}{3}) }{\sin ( \gamma^* - \pi)} = \kappa^*_1 \kappa_3^* \,,
\end{align*}
 and $c_1(v^{(4)}) = (\kappa_1^*)^2$. Therefore, taking into account that $\kappa_1^* < 0$
\begin{align*}
 \sum_{i=1}^3  c_i(v^4) \int_{\Gamma_i^*} {v^{(4)}}_{i} &= 2 l_1^* (\kappa_1^*)^2 + 2 l^*_2 \kappa^*_1 \kappa_2^* f(\gamma^*) + 2 l^*_3 \kappa^*_1 \kappa_3^* g(\gamma^*)\\[-9pt]
&>  2 l^*_2 \kappa^*_1 \kappa_2^* f(\gamma^*) + 2 l^*_3 \kappa^*_1 \kappa_3^* f(\gamma^*)= 2 \kappa^*_1 f(\gamma^*) \big(  l^*_2 \kappa_2^*  +  l^*_3 \kappa_3^* \big) \\
&= 2 \kappa^*_1 f(\gamma^*)(-(\gamma^* - \tfrac{\pi}{3}) - (\gamma^* - \pi))=  4 \kappa^*_1 f(\gamma^*)(-\gamma^* + \tfrac{2\pi}{3}) \\ 
& > 0\,.   
\end{align*}
\end{proof}
\begin{lem}\label{Lem: v5 negative}
Let $ 0< \gamma^* < \frac{2 \pi}{3}$.  Then 
\begin{align*}
\mathrm{(i)} & \;\int_{\Gamma_1^*} {v^{(5)}}_1 = 0  \,, \quad \int_{\Gamma_2^*} {v^{(5)}}_2 < 0 \,, \quad \int_{\Gamma_3^*} {v^{(5)}}_3 > 0 \,,
\\[2pt]
\mathrm{(ii)}& \;\int_{\Gamma_1^*} {v^{(5)}}_1 - \int_{\Gamma_2^*} {v^{(5)}}_2 > 0 \,, \quad \int_{\Gamma_2^*} {v^{(5)}}_2 - \int_{\Gamma_3^*} {v^{(5)}}_3 < 0 \,,\\
\mathrm{(iii)}& \; \sum_{i=1}^3  c_i(v^5) \int_{\Gamma_i^*} {v^{(5)}}_{i} > 0 \,.   
\end{align*}
\end{lem}
\begin{proof}
Let us first consider the case $\gamma^* = \frac{\pi}{3}$. Then  
\begin{align*}
\int_{-l^*_1}^{l^*_1} {v^{(5)}}_1 &=0, & &c_1(v^{(4)}) = 0\,,\\
\int_{-l^*_2}^{l^*_2} {v^{(5)}}_2 &= \tfrac{\kappa_1^*}{\sin ( \frac{\pi}{3})}  \int_{-l^*_2}^{l^*_2} \frac{1}{2}x^2 + \frac{3}{8}\frac{1}{(\kappa_1^*)^2}< 0 \,, & & c_2(v^{(5)}) = \tfrac{\kappa_1^*}{\sin ( \frac{\pi}{3})} < 0, \\
\int_{-l^*_3}^{l^*_3} {v^{(5)}}_3 &= \frac{-1}{ \kappa_1^*\sin ( \frac{\pi}{3})} \int_{- l^*_3}^{l^*_3}  \frac{1}{2} \cos(\kappa_1^* x) + 1>0 \,, & & c_3(v^{(5)}) =\frac{1}{2}  \frac{-(\kappa_1^*)^2}{ \kappa_1^*\sin ( \frac{\pi}{3})}>0 \,.
\end{align*}
Therefore, 
$$
\sum_{i=1}^3  c_i(v^{(5)}) \int_{\Gamma_i^*} \!{v^{(5)}}_i > 0 \qquad (\gamma^* = \tfrac{\pi}{3})\,.
$$
Assertion (ii) is an immediate consequence of assertion (i). This proves (i)-(iii) in case $\gamma^* = \frac{\pi}{3}$.

Now assume $\gamma^* \neq \frac{\pi}{3}$. Clearly  $\int_{\Gamma_1^*} {v^{(5)}}_1 = 0$ as ${v^{(5)}}_1 = 0$. Next we compute   
\begin{align*}
| \sigma|^2 &= \frac{1}{(\kappa_2^*)^2} +  \frac{\sin^2( \frac{\pi}{3})}{ \sin^2( \gamma^* - \frac{\pi}{3})}(r^*)^2 + 2r^* \frac{\sin( \frac{\pi}{3})}{ \sin( \gamma^* - \frac{\pi}{3})
}\cos(\kappa_2^* x)\frac{1}{\kappa_2^*} \\
&=\frac{\sin^2( \gamma^* + \frac{\pi}{3})}{\sin^2( \gamma^* - \frac{\pi}{3})}  (r^*)^2+  \frac{\sin^2( \frac{\pi}{3})}{ \sin^2( \gamma^* - \frac{\pi}{3})}(r^*)^2 - 2(r^*)^2 \frac{\sin( \frac{\pi}{3})\sin( \gamma^* + \frac{\pi}{3})}{ \sin^{2}( \gamma^* - \frac{\pi}{3})
}\cos(\kappa_2^* x) \\
& \geq (r^*)^2 \big(  \frac{\sin( \gamma^* + \tfrac{\pi}{3}) - \sin( \tfrac{\pi}{3})}{\sin( \gamma^* - \frac{\pi}{3})} \big)^2 \qquad \text{ for } \sigma \in \Gamma^*_2 \,.
\end{align*}
Therefore,
\begin{align*}
-\int_{\Gamma_2^*}\! {v^{(5)}}_2 &=  \int_{\Gamma_2^*}\!\partial_\gamma G_2( \sigma, r^*,\gamma^*) \\
&=\tfrac{1}{r^{*}\sin ( \gamma^* + \frac{\pi}{3})}\int_{\Gamma_2^*}\! \cos( \gamma^* - \tfrac{\pi}{3})|\sigma|^2 -(r^*)^2\cos (\gamma^* - \pi) \\ & \geq \tfrac{2l_{2}^*}{r^{*}\sin ( \gamma^* + \frac{\pi}{3})}(r^{*})^2 f(\gamma^*) \,,
\end{align*} 
where
$$  f(x):=\cos(x - \tfrac{\pi}{3}) \big(  \frac{\sin (x+ \tfrac{\pi}{3}) - \sin( \tfrac{\pi}{3})}{\sin(x- \frac{\pi}{3})} \big)^2 + \cos( x)\,.  
$$
It is not hard to  show that the function $f$ is strictly decreasing.  Together with the fact that  this function vanishes at $\gamma^* = \tfrac{2 \pi}{3} $ we conclude $f> 0$ and so  $\int_{\Gamma_2^*}\! {v^{(5)}}_2 < 0$. A similar proof works for ${v^{(5)}}_3$. 
This completes the proof of  (i). 

The statement (ii) is an immediate consequence of  assertion (i). 
Similarly you can check that $c_2(v^{(5)}) < 0$ and $c_3(v^{(5)})>0$ which easily gives  (iii). 
\end{proof} 
\section{Deriving  the parabolic  system }\label{App: nonlocal}

For the normal velocity $V_i$ of $\Gamma_i(t):= \Gamma_{\rho_i, \mu_i}(t)$ we obtain with the convention  \eqref{DB:x and sigma} 
\begin{equation*}
\begin{aligned}
V_i (\sigma,t) &= \langle \partial_t  \Phi_i (\sigma, t),  n_i(\sigma, t) \rangle \\
&= \frac{1}{J_i}\big\langle \partial_w \Psi_i, R \partial_\sigma \Psi_i \big\rangle \partial_t \rho_i( \sigma, t) +  \big \langle \partial_r \Psi_i, n_i(\sigma , t)  \big \rangle \partial_t \mu_i ( \mathrm{pr}_i (\sigma) , t )  \quad \sigma \in \Gamma_i^*, 
\end{aligned}
\end{equation*}
where the unit normal
$n_i$ of $\Gamma_i(t) := \Gamma_{\rho_i, \mu_i}(t)$ is given by\begin{align} \label{Eq: D normal}
n_{i}( \sigma ,t ) &= \frac{1}{J_i} R \partial_\sigma \Phi_i(\sigma ,t) \nonumber \\ 
               &= \frac{1}{J_i} \big(R \partial_\sigma \Psi_i + R \partial_w \Psi_i \,\partial_\sigma \rho_i (\sigma ,t)\big) \quad \sigma \in \Gamma_i^* \,.
\end{align}
Here
\begin{equation*} 
J_i=J_i(\sigma, \rho_i , \mu_i ) := \abs{ \partial_\sigma \Phi_i} = \sqrt{\abs{ \partial_\sigma \Psi_i}^2 + 2 \langle \partial_\sigma \Psi_i, \partial_w \Psi_i \rangle \partial_\sigma \rho_i  + \abs{\partial_w \Psi_i}^2 ( \partial_\sigma \rho_i )^2},
\end{equation*}
and $R$ denotes the anti-clockwise rotation by $\pi/2$. Next computing  the curvature $\kappa_i( = \kappa_i (\sigma, \rho_i , \mu_i))$ of $\Gamma_i(t) : = \Gamma_{\rho_i, \mu_i}(t)$ we get
\begin{align} \label{A curvature}
     \kappa_i &= \frac{1}{{J_i}^3} \big \langle \partial_\sigma^2 \Phi_i , R \partial_\sigma \Phi_i \big \rangle \\
       &= \frac{1}{{J_i}^3} \Big [ \langle  \partial_w \Psi_i, R \partial_\sigma \Psi_i \rangle \partial_\sigma^2\rho_i + \big\{2 \big \langle \partial_{\sigma w}\Psi_i, R \partial_\sigma \Psi_i  \big \rangle +  \big \langle \partial_{\sigma}^2\Psi_i, R \partial_w \Psi_i \big \rangle \big \}\partial_\sigma \rho_i \notag \\
       &\hspace{42pt}+  \big \{ \big  \langle\partial_{w w}\Psi_i , R \partial_\sigma \Psi_i \big \rangle  + 2 \big \langle \partial_{\sigma w}\Psi_i, R \partial _w \Psi_i  \big \rangle \notag\\
& \hspace{4.5cm}+ \big \langle \partial_{ww}\Psi_i,  R \partial_w \Psi_i  \big \rangle  \partial_\sigma \rho_i \big \} ( \partial_\sigma \rho_i)^2   + \langle \Psi_{\sigma \sigma}, R \Psi_\sigma \rangle \Big ].\notag
 \end{align}
Therefore, the surface diffusion flow equations  can be reformulated as 
\begin{equation} \label{Eq: D SDF P}
 \partial_t \rho_i = a_i(\sigma, \rho_i , \mu_i ) \Delta (\sigma, \rho_i , \mu_i ) \kappa_i (\sigma, \rho_i , \mu_i ) + b_i(\sigma, \rho_i , \mu_i ) \partial_t \mu_i \,,  
\end{equation}
where  
\begin{align*}
a_i(\sigma, \rho_i , \mu_i ) &:= \frac{J_i(\sigma,\rho_{i}, \mu_{i})}{ \big\langle \partial_w \Psi_i , R \partial_\sigma \Psi_i \big \rangle} \:\Big( = \frac{1}{ \big \langle n_i^*(\sigma) , n_i(\sigma, t) \big \rangle}\Big)\,,\\[8pt]
b_i(\sigma,\rho_i , \mu_i ) &:=  \frac{\big \langle \partial_r \Psi_i , R \partial_ \sigma \Psi_i \big \rangle + \big \langle \partial_r \Psi_i, R \partial_w \Psi_i \big \rangle \partial_\sigma \rho_i}{-\big \langle \partial_w \Psi_i, R \partial_\sigma \Psi_i \big  \rangle} \Big(= \frac{\big \langle \tau_i^*(\sigma) , n_i ( \sigma , t ) \big \rangle}{\big \langle n_i^*(\sigma) , n_i(\sigma, t) \big \rangle} \, \, \Big) \,,\\[8pt]
 \quad  \Delta(\sigma, \rho_i ,\mu_i) v &:=  \frac{1}{J_i(\sigma, \rho_i, \mu_i)} \partial_\sigma \Big (  \frac{1}{J_i(\sigma, \rho_i, \mu_i)} \partial_ \sigma v \Big).
\end{align*}
Note that we have omitted the projection $\mathrm{pr}_i$ in the functions $\mu_i$  and the term   $ ( \sigma, \rho_i ( \sigma,t) , \mu_i ( \mathrm{pr}_i(\sigma),t)) $ in $ \partial_u \Psi_i  $ with $u \in \{ \sigma, w, \mu  \}$ to shorten the formulas. Furthermore note 
\begin{equation} \label{b_i}
b_i|_{\rho_i \equiv 0} = - \langle \tau^*_i , \pm n^*_i \rangle = 0 \,, \qquad a_i|_{\rho_i \equiv 0} = 1\,.
\end{equation}
We will  now  make use of the linear dependency  \eqref{Eq: matrix rho} to derive from the equations \eqref{Eq: D SDF P} evolution equations solely for the functions $\rho_i$. For  this, let us    rewrite \eqref{Eq: D SDF P} into
\begin{equation}\label{Eq1: suitable form of rho}
\partial_t \rho_i = \mathfrak{F}_i( \rho_i ,  \rho|_{\Sigma^*}) +  \mathfrak{B}_i ( \rho_i ,   \rho|_{\Sigma^*})\partial_t \big( \mathcal J \rho \circ \mathrm{pr}_i\big)_i  \quad \text{ in } \Gamma^*_i \,,
\end{equation}
where for $\sigma \in \Gamma_i^*$
\begin{align*}
\mathfrak{F}_i( \rho_i ,  \rho|_{\Sigma^*})( \sigma) &= a_i \big(\sigma, \rho_i ,  ( \mathcal J \rho|_{\Sigma^*})_i \big) \Delta \big( \sigma, \rho_i , ( \mathcal J \rho|_{\Sigma^*})_i \big) \kappa_i \big(\sigma, \rho_i , ( \mathcal J \rho|_{\Sigma^*})_i \big), \\
  \mathfrak{B}_i ( \rho_i ,   \rho|_{\Sigma^*})(\sigma) &= b_i \big( \sigma, \rho_i ,( \mathcal J \rho|_{\Sigma^*})_i \big), 
\end{align*}
and where we  used the linear dependency  \eqref{Eq: matrix rho}. By writing \eqref{Eq1: suitable form of rho} as a vector identity on $\Sigma^*$ we get 
\begin{equation}\label{Eq: suitable form of rho}
\partial_t \rho = \mathfrak{F}( \rho ,\rho|_{\Sigma^*}) +  \mathfrak{B} ( \rho ,   \rho|_{\Sigma^*}) \mathcal J  (\partial_t \rho)  \quad \text{ on } \Sigma^* \,,
\end{equation}
where we  employed the following notations
\begin{align*}
 \mathfrak{F}( \rho ,  \rho|_{\Sigma^*})(\sigma)    &:= \Big( \mathfrak{F}_i( \rho_i ,  \rho|_{\Sigma^*})( \sigma) \Big)_{i = 1,2,3} \quad \text{ for } \sigma \in \Sigma^* \,, \\
  \mathfrak{B} ( \rho ,   \rho|_{\Sigma^*})(\sigma) &:= \mathrm{diag} \Big(  \big(\mathfrak{B}_i ( \rho_i ,   \rho|_{\Sigma^*})(\sigma)\big)_{i =1, 2, 3} \Big) \quad \text{ for } \sigma \in \Sigma^* \,.   
\end{align*}
We rearrange to find
\begin{equation} \label{Eq: DB F(rho)} 
\big( I - \mathfrak{B} ( \rho ,\rho|_{\Sigma^*}) \mathcal J \big) \partial_t \rho = \mathfrak{F}( \rho ,  \rho|_{\Sigma^*}) \quad \text{ on } \Sigma^* \,.  
\end{equation}
Consequently we get 
$$
\partial_t \rho = \big( I - \mathfrak{B} ( \rho ,\rho|_{\Sigma^*}) \mathcal J \big)^{-1} \mathfrak{F}( \rho ,  \rho|_{\Sigma^*}) \quad \text{ on } \Sigma^*.
$$
According to  \eqref{b_i}, in some neighborhood of $\rho \equiv 0$  in $C^1(\Gamma^*)$ the  inverse $\big( I - \mathfrak{B} ( \rho ,\rho|_{\Sigma^*}) \mathcal J \big)^{-1}$ exists.
Inserting the above equation  into the equation \eqref{Eq: suitable form of rho} we  can finally reformulate the surface diffusion flow equations 
$$ 
V_i  = -\Delta_{\Gamma_i} \kappa_i  \quad \text{ on } \Gamma_i(t)
$$ as a   system of the  evolution equations  for  functions $\rho_i$ defined on fixed domains $ \Gamma_i^*$ (or equivalently on  $[-l_i^*, l_i^*]$)   
$$
\partial_t \rho_i = \mathfrak{F}_i( \rho_i ,  \rho|_{\Sigma^*}) +  \mathfrak{B}_i ( \rho_i ,   \rho|_{\Sigma^*}) \Big(  \big \{ \mathcal J  \big( I - \mathfrak{B} ( \rho ,\rho|_{\Sigma^*}) \mathcal J \big)^{-1} \mathfrak{F}( \rho ,  \rho|_{\Sigma^*})  \big \} \circ \mathrm{pr}_i \Big)_i. 
$$

Finally, we rewrite the boundary conditions at $\sigma \in \Sigma^*$ as
\begin{equation*}
\begin{aligned} 
\mathfrak{G}_1( \rho) ( \sigma)&:= \rho_1 ( \sigma) + \rho_2 (\sigma) + \rho_3( \sigma) = 0 \,,  \\[6pt]
\mathfrak{G}_2( \rho) (\sigma) &: =  \langle n_1( \sigma) , n_2( \sigma) \rangle - \cos  \tfrac{2 \pi}{3} \\
&\,{} =   \big \langle    \frac{1}{J_1} ( \partial_\sigma \Psi_1 +  \partial_w \Psi_1 \,\partial_\sigma \rho_1 ) \,,\,  \frac{1}{J_2} ( \partial_\sigma \Psi_2 +  \partial_w \Psi_2 \,\partial_\sigma \rho_2 ) \big\rangle - \cos  \tfrac{2 \pi}{3} = 0 \,, \\[6pt] 
\mathfrak{G}_3 ( \rho) (\sigma) &: =  \langle n_2 ( \sigma), n_3( \sigma) \rangle - \cos  \tfrac{2 \pi}{3} \\
&\,{} =   \big \langle    \frac{1}{J_2} (\partial_\sigma \Psi_2 +  \partial_w \Psi_2 \,\partial_\sigma \rho_2 ) \,,\,  \frac{1}{J_3} ( \partial_\sigma \Psi_3 +  \partial_w \Psi_3 \,\partial_\sigma \rho_3 ) \big\rangle - \cos  \tfrac{2 \pi}{3} = 0 \,,\\
\mathfrak G_4 (\rho)( \sigma ) &:= \sum_{i = 1}^3\kappa_i \big(\sigma, \rho_i , ( \mathcal J \rho|_{\Sigma^*})_i \big) =0 \,,   \\
\mathfrak  G_5(\rho)( \sigma ) &:=\frac{1}{J_1} \partial_ \sigma \big(\kappa_1
\big(\sigma, \rho_1 , ( \mathcal J \rho|_{\Sigma^*})_1 \big) - \frac{1}{J_2} \partial_ \sigma \big(\kappa_2 \big(\sigma, \rho_2 , ( \mathcal J \rho|_{\Sigma^*})_2 \big) = 0 \,,  \\
\mathfrak G_6 ( \rho) ( \sigma) &:= \frac{1}{J_2} \partial_ \sigma \big(\kappa_2 \big(\sigma, \rho_2 , ( \mathcal J \rho|_{\Sigma^*})_2 \big) - \frac{1}{J_3} \partial_ \sigma \big(\kappa_3 \big(\sigma, \rho_3 , ( \mathcal J \rho |_{\Sigma^*})_3 \big) = 0  \,. 
\end{aligned}
\end{equation*}
We emphasize that  the operators $\mathfrak{G}_i$ $(i = 1,\dots , 5)$ are purely local due to the fact that the projections $\mathrm{pr}_i$ act as the identity on their image $\Sigma^*$.

\end{appendices}

\bibliographystyle{plain}
\bibliography{paper}

\begin{thebibliography}{10}

\bibitem{AbelsArabGarcke}
H.~Abels, N.~Arab, and H.~Garcke.
\newblock On convergence of solutions to equilibria for fully nonlinear
  parabolic systems with nonlinear boundary conditions.
\newblock submitted for publication, arXiv: 1403.4526.

\bibitem{AbsilKurdyka}
P.~A. Absil and K.~Kurdyka.
\newblock On the stable equilibrium points of gradient systems.
\newblock {\em Systems Control Lett.}, 55(7):573--577, 2006.

\bibitem{Depner}
D.~Depner.
\newblock {\em Stability Analysis of Geometric Evolution Equations with Triple
  Lines and Boundary Contact}.
\newblock PhD thesis, Regensburg, 2010.
\newblock urn:nbn:de:bvb:355-epub-160479.

\bibitem{DepnerGarckelinearized}
D.~Depner and H.~Garcke.
\newblock Linearized stability analysis of surface diffusion for hypersurfaces
  with triple lines.
\newblock {\em Hokkaido Math. J.}, 42(1):11--52, 2013.

\bibitem{DepnerGarckeKohsaka}
D.~Depner, H.~Garcke, and Y.~Kohsaka.
\newblock Mean curvature flow with triple junctions in higher space dimensions.
\newblock {\em Arch. Rational Mech. Anal.}, 211(1):301 -- 334, 2014.

\bibitem{Escher-Mayer-Simonett}
J.~Escher, U.~F. Mayer, and G.~Simonett.
\newblock The surface diffusion flow for immersed hypersurfaces.
\newblock {\em SIAM J. Math. Anal.}, 29(6):1419--1433 (electronic), 1998.

\bibitem{Escher-Simonett}
J.~Escher and G.~Simonett.
\newblock The volume preserving mean curvature flow near spheres.
\newblock {\em Proc. Amer. Math. Soc.}, 126(9):2789--2796, 1998.

\bibitem{FoisyAlfaroBrock...}
J.~Foisy, M.~Alfaro, J.~Brock, N.~Hodges, and J.~Zimba.
\newblock The standard double soap bubble in {${\bf R}^2$} uniquely minimizes
  perimeter.
\newblock {\em Pacific J. Math.}, 159(1):47--59, 1993.

\bibitem{HG}
H.~Garcke.
\newblock Curvature driven interface evolution.
\newblock {\em Jahresber. Dtsch. Math.-Ver.}, 115(2):63--100, 2013.

\bibitem{GarckeItoKohsakatriplejunction}
H.~Garcke, K.~Ito, and Y.~Kohsaka.
\newblock Surface diffusion with triple junctions: a stability criterion for
  stationary solutions.
\newblock {\em Adv. Differential Equations}, 15(5-6):437--472, 2010.

\bibitem{GarckeNovick}
H.~Garcke and A.~Novick-Cohen.
\newblock A singular limit for a system of degenerate {C}ahn-{H}illiard
  equations.
\newblock {\em Adv. Differential Equations}, 5(4-6):401--434, 2000.

\bibitem{HutchingsMorganRitorAntonio}
M.~Hutchings, F.~Morgan, M.~Ritor{\'e}, and A.~Ros.
\newblock Proof of the double bubble conjecture.
\newblock {\em Ann. of Math. (2)}, 155(2):459--489, 2002.

\bibitem{Mayer-SD}
U.~F. Mayer.
\newblock Numerical solutions for the surface diffusion flow in three space
  dimensions.
\newblock {\em Comput. Appl. Math.}, 20(3):361--379, 2001.

\bibitem{MorganWichiramala}
F.~Morgan and W.~Wichiramala.
\newblock The standard double bubble is the unique stable double bubble in
  {$\bold R^2$}.
\newblock {\em Proc. Amer. Math. Soc.}, 130(9):2745--2751 (electronic), 2002.

\bibitem{Prokert}
G.~Prokert.
\newblock {\em Parabolic evolution equations for quasistationary free boundary
  problems in capillary fluid mechanics}.
\newblock Dissertation, Technische Universiteit Eindhoven, Eindhoven, 1997.

\bibitem{Pruss20093902}
J.~Prüss, G.~Simonett, and R.~Zacher.
\newblock On convergence of solutions to equilibria for quasilinear parabolic
  problems.
\newblock {\em Journal of Differential Equations}, 246(10):3902 -- 3931, 2009.

\bibitem{pruss2009612}
J.~Prüss, G.~Simonett, and R.~Zacher.
\newblock On normal stability for nonlinear parabolic equations.
\newblock {\em Discrete Contin. Dyn. Syst.}, (Dynamical Systems, Differential
  Equations and Applications. 7th AIMS Conference, suppl.):612--621, 2009.

\bibitem{RunstSickel}
T.~Runst and W.~Sickel.
\newblock {\em Sobolev spaces of fractional order, {N}emytskij operators, and
  nonlinear partial differential equations}, volume~3 of {\em de Gruyter Series
  in Nonlinear Analysis and Applications}.
\newblock Walter de Gruyter \& Co., Berlin, 1996.

\bibitem{Taylor-Cahn}
J.~E. Taylor and J.~W. Cahn.
\newblock Linking anisotropic sharp and diffuse surface motion laws via
  gradient flows.
\newblock {\em J. Statist. Phys.}, 77(1-2):183--197, 1994.

\bibitem{Zeidlernonlinear1}
E.~Zeidler.
\newblock {\em Nonlinear functional analysis and its applications. {I}}.
\newblock Springer-Verlag, New York, 1986.
\newblock Fixed-point theorems, Translated from the German by Peter R. Wadsack.

\end{thebibliography}
\end{document}